\documentclass{daj}

\dajAUTHORdetails{%
  title = {Asymptotic Structure for the Clique Density Theorem}, 
  author = {Jaehoon Kim, Hong Liu, Oleg Pikhurko, and Maryam Sharifzadeh},
  plaintextauthor = {Jaehoon Kim, Hong Liu, Oleg Pikhurko, and Maryam Sharifzadeh},
  plaintexttitle = {Asymptotic Structure for the Clique Density Theorem}, 
  runningtitle = {Asymptotic Structure for the Clique Density Theorem}, 
  runningauthor = {Jaehoon Kim, Hong Liu, Oleg Pikhurko, and Maryam Sharifzadeh},
  copyrightauthor = {J. Kim, H. Liu, O. Pikhurko, and M. Sharifzadeh},
  keywords = {graph limits, graphon, clique density theorem, stability, etc.},
}   

\dajEDITORdetails{%
   year={2020},
   number={19},
   received={17 August 2019},   
   published={30 December 2020},  
   doi={10.19086/da.18559},       
}   

\usepackage{amsthm,amssymb,amsmath,enumerate,graphicx,psfrag,enumitem}

\allowdisplaybreaks
\usepackage[textsize=footnotesize]{todonotes}

\usepackage{algorithm}
\usepackage{tikz}

\usetikzlibrary{backgrounds}

\newtheorem{definition}{Definition}[section]
\newtheorem{claim}{Claim}

\newtheorem{theorem}[definition]{Theorem}

\newtheorem{lemma}[definition]{Lemma}
\newtheorem{conjecture}[definition]{Conjecture}

\newtheorem{remark}[definition]{Remark}

\numberwithin{equation}{section}

\newcommand{\comment}[1]{}
\newcommand{\N}{\mathbb N}

\newcommand{\cH}{\mathcal{H}}

\newcommand{\bI}{\mathbf{I}}

\newcommand{\dd}{\,\mathbf{d}}

\newcommand{\COMMENT}[1]{}
\renewcommand{\COMMENT}{\footnote} 

\newcounter{step}

\newcommand{\C}[1]{\mathcal{#1}}
\newcommand{\e}{\varepsilon}
\newcommand{\hide}[1]{}
\renewcommand{\mid}{:}
\newcommand{\I}[1]{\mathbb{#1}}

\renewcommand{\ge}{\geqslant}
\renewcommand{\le}{\leqslant}
\newcommand{\random}[2]{\I G(#1,#2)}

\begin{document}

\begin{frontmatter}[classification=text]

\title{Asymptotic Structure for the Clique Density Theorem} 

\author[Jaehoon]{Jaehoon Kim\thanks{Supported by the Leverhulme Trust Early Career Fellowship~ECF-2018-538, by the POSCO Science
Fellowship of POSCO TJ Park Foundation, and by the KAIX Challenge program of KAIST Advanced Institute
for Science-X. }}
\author[Hong]{Hong Liu\thanks{Supported by the UK Research and Innovation Future Leaders Fellowship MR/S016325/1 and the Leverhulme Trust Early Career Fellowship~ECF-2016-523.} }
\author[Oleg]{Oleg Pikhurko\thanks{Supported by the Leverhulme Research Project Grant RPG-2018-424.}}
\author[Maryam]{Maryam Sharifzadeh\thanks{Supported by the European Unions Horizon 2020 research and innovation programme under the Marie Curie Individual Fellowship agreement No 752426. } }
\begin{abstract}
The famous Erd\H os-Rademacher problem asks for the smallest number of $r$-cliques in a graph with the given number of vertices and edges. Despite decades of active attempts, the asymptotic value of this extremal function for all $r$ was determined only recently, by Reiher [\emph{Annals of Mathematics}, \textbf{184} (2016) 683--707]. Here we describe the asymptotic structure of all almost extremal graphs. This task for $r=3$ was previously accomplished by Pikhurko and Razborov [\emph{Combinatorics, Probability and Computing}, \textbf{26} (2017) 138--160].
\end{abstract}
\end{frontmatter}

\section{Introduction}
Let $K_r$ denote the complete graph on $r$ vertices and let the \emph{Tur\'an graph} $T_{r}(n)$ be the complete $r$-partite graph with $n$ vertices and balanced part sizes (that is, every two parts differ in size by at most 1). 

It is fair to say that extremal graph theory was born with the following fundamental theorem of Tur\'an~\cite{Turan41}: among all graphs on $n$ vertices without $K_r$,
the Tur\'an graph $T_{r-1}(n)$ is the unique (up to isomorphism) graph with the maximum number of edges. The case $r=3$ of this theorem was proved earlier by Mantel~\cite{Mantel07}. 

Rademacher (unpublished, see e.g.~\cite{Erdos55,Erdos62}) showed that, for even $n$, any $n$-vertex graph with $t_2(n)+1$ edges contains at least $n/2$ triangles, where $t_r(n):=e(T_r(n))$ is the number of edges in the $r$-partite Tur\'an graph $T_r(n)$. In 1955, 
Erd\H os~\cite{Erdos55} asked  the more general question to determine \emph{$G_r(n,m)$}, the minimum number of copies of $K_r$ in an \emph{$(n,m)$-graph}, that is, a graph with $n$ vertices and $m$ edges.
This question is now known as the \emph{Erd\H os-Rademacher problem}. Early papers on this problem (\cite{Erdos55,Erdos62,MoonMoser62,Nikiforov76,NordhausStewart63}, etc) dealt mainly with the case when $m$ is slightly larger than $t_{r-1}(n)$, the threshold when copies of $K_r$ start to appear. But even this special case turned out to be quite difficult. For example, the conjecture of Erd\H os~\cite{Erdos55} that $G_3(n,t_2(n)+q)\ge q\,\lfloor n/2\rfloor$ for $q< \lfloor n/2\rfloor$ when $n$ is large was proved only two decades later by Lov\'asz and Simonovits~\cite{LovaszSimonovits76}, with a proof of the conjecture also announced by Nikiforov and Khadzhiivanov~\cite{NikiforovKhadzhiivanov81}.

Lov\'asz and Simonovits~\cite[Conjecture~1]{LovaszSimonovits76} made the following bold conjecture. Let $\C H$ consist of all graphs that can be obtained from a complete partite graph by adding a triangle-free graph into one of the parts. Let $H_r(n,m)$ be the miminum number of $r$-cliques in an $(n,m)$-graph
from $\C H$. Clearly, $G_r(n,m)\le H_r(n,m)$. In this notation, the conjecture of Lov\'asz and Simonovits states that we have equality here, provided $n$ is sufficiently large:

\begin{conjecture}[Lov\'asz and Simonovits~\cite{LovaszSimonovits76}]
	\label{cj:LS76}
	For every integer $r\ge 3$, there is $n_0$ such that for all $n\ge n_0$ and $0\le m\le {n\choose 2}$, we have that
	$G_r(n,m)\ge H_r(n,m)$.
	\end{conjecture}

Of course, minimising the number of $r$-cliques over $(n,m)$-graphs from the restricted class $\C H$ is easier than the unrestricted version. The computation of $H_3(n,m)$ for all $(n,m)$ appears in~\cite[Proposition~1.5]{LiuPikhurkoStaden17arxiv}.  Some large ranges of parameters when the conjecture has been proved are when $m$ is slightly above $t_r(n)$ by Lov\'asz and Simonovits~\cite{LovaszSimonovits83} and when $r=3$ and $m/{n\choose 2}$ is bounded away from 1 by Liu, Pikhurko and Staden~\cite{LiuPikhurkoStaden17arxiv}. Still, Conjecture~\ref{cj:LS76} remains open.

Let us turn to the asymptotic version. Namely, given $\alpha \in [0,1]$ take any integer-valued function $0\le e(n)\le {n\choose 2}$ with $e(n)/{n\choose 2}\to \alpha$ as $n\to\infty$ and define
 \begin{eqnarray*}
g_r(\alpha) &:=& \lim_{n \rightarrow \infty}\frac{G_r(n,e(n))}{\binom{n}{r}},\\
h_r(\alpha) &:=& \lim_{n \rightarrow \infty}\frac{H_r(n,e(n))}{\binom{n}{r}}.
 \end{eqnarray*}
 It is not hard to see from basic principles (see e.g.~\cite[Lemma~2.2]{PikhurkoSliacanTyros19}) that the limits exist and do not depend on the choice of the function~$e(n)$. Thus, the determination of $g_r(\alpha)$ amounts to estimating the Erd\H os-Rademacher function within additive error $o(n^r)$. Clearly, $h_r(\alpha)$ is an upper bound on~$g_r(\alpha)$. 
 
For each $\alpha\in [0,1]$, it is not hard to find some sequence of graphs $(H_{\alpha,n})_{n\in\mathbb N}$, that give the value of $h_r(\alpha)$. If $\alpha=1$, then we can let 
$H_{1,n}:=K_n$ be the complete graph. Suppose that $0\le \alpha<1$. Let integer $k \ge1$ satisfy
$\alpha\in [1-\frac{1}{k},1-\frac1{k+1})$. Then fix the (unique) $c> \frac1{k+1}$ so that the complete $(k+1)$-partite graph $H_{\alpha,n}:=K(V_1,\dots,V_{k+1})$ with parts $V_1,\dots,V_{k+1}$, where $|V_1|=\dots=|V_{k}|=\lfloor cn\rfloor$,
has edge density $(\alpha+o(1)){n\choose 2}$ as $n\to\infty$. It is easy to show, see~\eqref{eq:c},
that  $c\le \frac1{k}$. Thus $|V_{k+1}|=n- k\lfloor{cn}\rfloor\ge 0$. (In fact, our choice to round
$cn$ down was rather arbitrary, just to make the graph $H_{\alpha,n}$ well-defined for each~$\alpha$ and $n$.)%
It is routine to show (see e.g.~\cite[Theorem~1.3]{Nikiforov11} for a derivation) that these ratios give the value of $h_r(\alpha)$, that is,
 \begin{equation}\label{eq:hr}
h_r(\alpha)=r!\left({k\choose r} c^r + {k\choose r-1}c^{r-1}(1- kc)\right).
\end{equation}
 The function $h_r$ stays 0 when $\alpha\le 1-\frac1{r-1}$ (when $k\le r-2$). Also,
 as Lemma~\ref{eq:  h'} implies, $h_r$ consists of countably many concave ``scallops'' with cusps at $\alpha=1-\frac1m$ for integer $m\ge r-1$.

For a while, the best known lower bound on the limit function $h_r $, by Bollob\'as~\cite{Bollobas76}, was the piecewise linear function which coincides with $h_r$ on the cusp points. 
Fisher~\cite{Fisher89} showed that $g_3(\lambda)=h_3(\lambda)$ for all $1/2\le \lambda\le 2/3$, that is, he determined $g_3 $ in the first scallop.
Razborov used his newly developed theory of flag algebras first to give a different proof of Fisher's result in~\cite{Razborov07} and then to determine the whole function $g_3 $ in~\cite{Razborov08}. The function $g_4 $ was determined by Nikiforov~\cite{Nikiforov11} and the function $g_r $ for any $r\ge 5$ was determined by Reiher~\cite{Reiher16} (with these two papers also giving new proofs for the previously solved cases of~$r$).

\subsection{Main result}
We are interested in the asymptotic structure of (almost) extremal graphs for the Erd\H os-Rademacher $K_r$-minimisation problem, that is, a description up to $o(n^2)$ edges of every $(n,m)$-graph with $G_r(n,m)+o(n^r)$ copies of $K_r$ as $n\to\infty$. Of course, such a result tells us more about the problem than just the value of $g_r $. Asymptotic structure results are often very useful for proving enumerative and probabilistic versions of the corresponding extremal problem. For example, the more general problem of understanding the structure of graphons with any given edge and $r$-clique densities appears in the study of  exponential random graphs (see
Chatterjee and Diaconis~\cite{ChatterjeeDiaconis13} and its follow-up papers), phases in large graphs (see the survey by Radin~\cite{Radin18aihpd}),
and large deviation inequalities for the clique density
(see the survey by Chatterjee~\cite{Chatterjee16bams}).
Last but not least, asymptotic structure results often greatly help, as a first step, in obtaining the exact structure of extremal graphs via the so-called \emph{stability approach} pioneered by Simonovits~\cite{Simonovits68}. Here, knowing extremal $n$-vertex graphs within $o(n^2)$ edges greatly helps in the ultimate aim of ruling out even a single ``wrong'' adjacency.
In fact, almost all  cases when the Erd\H os-Rademacher problem was solved exactly were established via the stability approach.

In order to state the main result of this paper, we have to define further graph families. 
For $\alpha=1$, we let $\C H_{1,n} :=\{K_n\}$. 
For $\alpha\in [0,1)$, by using the notation defined before~\eqref{eq:hr}, let $\C H_{\alpha,n}$ consist of all graphs that are obtained from the complete $k$-partite graph on parts $V_1,\dots,V_{k-1}$ and $U:=V_{k}\cup V_{k+1}$ by adding a triangle-free graph on $U$ with $|V_{k}|\cdot|V_{k+1}|$ edges. 
(In particular, $\C H_{0,n} :=\{\,\overline{K_n}\,\}$ consists of the empty graph only.) Clearly, for any $r\ge 3$, the number of $r$-cliques in the obtained graph does not depend on the choice of the graph added on $U$. Also, $\C H_{\alpha,n}\ni H_{\alpha,n}$ is always non-empty (but typically has many non-isomorphic graphs). Finally, for $r\ge 3$, let $\C H_{r,n}$ be the union of $\C H_{\alpha,n}$ over all $\alpha\in [0,1]$ together with the family of all $K_r$-free $n$-vertex graphs.

Pikhurko and Razborov~\cite{PikhurkoRazborov17} proved that every almost extremal $(n,m)$-graph $G$ is $o(n^2)$ close in edit distance to some graph in $\C H_{3,n}$. Our main result is to extend this structural result to all $r\ge 4$:

\begin{theorem}\label{th:graph} For every real $\e>0$ and integer $r\ge 4$, there are $\delta>0$ and $n_0$ such that every graph $G$ with $n\ge n_0$ vertices and at most $(g_r(\alpha)+\delta){n\choose r}$ $r$-cliques, where $\alpha:=e(G)/{n\choose 2}$, can be made isomorphic to some graph in $\C H_{r,n}$ by changing at most $\e n^2$ adjacencies. 
\end{theorem}

Note that $\C H_{3,n}\subseteq \C H_{r,n}$. Also, all graphs in $\C H_{r,n}\setminus \C H_{3,n}$ are $K_r$-free but may contain triangles; these are ``trivial'' minimisers for the $g_r$-problem that need not be minimisers for the $g_3$-problem. Thus, apart from these graphs,  the $g_3$ and $g_r$ extremal problems have the same set of approximate minimisers and we exploit this in our proof as follows.
In brief, we take any almost $G_r(n,m)$-extremal graph~$G$. Suppose that $G$ has strictly more than $G_3(n,m)+o(n^3)$ triangles for otherwise $G$ is $o(n^2)$-close in edit distance to $\C H_{3,n}$ by the result in~\cite{PikhurkoRazborov17} and we are done since $\C H_{3,n}\subseteq \C H_{r,n}$. If $\alpha:= m/{n\choose 2}$ is $1-\frac1k+o(1)$ for some integer $k\ge r$ (that is, the edge density of $G$ is close to that of some Tur\'an graph $T_k(n)$), then we use the result of Lov\'asz and Simonovits~\cite{LovaszSimonovits83} that $G$ has to be $o(n^2)$-close in edit distance to $T_k(n)$, giving the desired conclusion. Thus we can assume that the edge density is strictly inside one of the scallops. Lemma~\ref{eq:  h'} shows that the function $h_r $ is differentiable for such $\alpha$.
This allows us to derive various properties of $G$ via variational principles. The property that we will need is that, for a typical vertex $x$ of $G$, there is an asymptotic linear relation between the degree of $x$ and the number of $r$-cliques containing $x$. This relation comes from the Lagrange multiplier method. Since we know the extremal function $g_r $, we can determine all Lagrange multipliers and write an explicit relation. Since
the graph $G$ is ``heavy'' on triangles, we can find a typical vertex $x$ that is ``heavy'' in terms of  triangles containing it. When we restrict ourselves to the graph $G'$ induced by the set of neighbours of $x$, then the counts of triangles and $r$-cliques in $G$ containing $x$ correspond to the counts of respectively  edges and $(r-1)$-cliques in $G'$. Some calculations show that $G'$ is too ``heavy'' on $K_2$ when compared to the number of $(r-1)$-cliques, contradicting the asymptotic result for $r-1$ and finishing the proof. Thus, the main results on which our proof of Theorem~\ref{th:graph} for a given $r\ge 4$ crucially relies are the values of $g_r $
and $g_{r-1} $ as well as the asymptotic structure for $r=3$.

We found it more convenient to present our proof in terms of \emph{graphons} that are analytic objects representing subgraph densities in large dense graphs. This reduces the number of parameters in various statements. For example, the statement that some ``natural'' property fails for $o(n)$ vertices corresponds in the limit to the statement that the set of failures has measure 0. Also, the variational principles are easier to state and derive using the graphon language, in particular it is much cleaner to define the limit versions of the families $\C H_{r,n}$, namely the families $\C H_r$ defined in Section~\ref{sec-graphon-extremal-structure}. Some downside of this is that we have to use various non-trivial (but standard) facts of measure theory. 
We rectify this by giving discrete analogues of some analytic constructions and properties that we use. 
Also, we believe that graphons, as a tool in extremal graph theory, are by now standard and widely known.

\medskip

\noindent\textbf{Organisation.} The rest of the paper is organised as follows. We will rephrase our main result in terms of the structure of extremal graphons in the next subsection, Section~\ref{sec-graphon-extremal-structure}. Further properties  of graphons and of the family of extremal graphons are discussed in Sections~\ref{graphons} and~\ref{subsec: hr}. This is preceeded by Section~\ref{MT} that contains some notions and results of measure theory that we will need. In Section~\ref{sec-graphon-to-graph}, we
derive Theorem~\ref{th:graph} from our result on graphons. Then in Section~\ref{sec-main-proof}, we present the proof of the graphon version of our main result.


\subsection{Graphons with minimum clique density}\label{sec-graphon-extremal-structure}
For an introduction to graphons, we refer the reader to the excellent book by Lov\'asz~\cite{Lovasz:lngl}.

For the purposes of this paper,
it is convenient to define a 
 \emph{graphon} as a pair $(W,\mu)$, where $W: [0,1]\times [0,1] \to [0,1]$ is a symmetric Borel function and $\mu$ is a non-atomic probability measure on Borel
subsets of~$[0,1]$. 
By small abuse of notation, we may call just the function $W$ a \emph{graphon} (if the measure $\mu$ is understood).
Each graph $G=(V,E)$ with $V=\{v_1,\ldots,v_n\}$ corresponds naturally to a graphon $(W_G,\mu)$ with $\mu$ being the Lebesgue measure on $[0,1]$ and $W_G$ being the \emph{adjacency function} of $G$ which assumes value 1 on $[\frac{s-1}n,\frac{s}n)\times [\frac{t-1}n,\frac{t}n)$ for each $\{v_s,v_t\}\in E(G)$ and 0 otherwise. \hide{
$W_{G}=(x_1,x_2)$ being $1$ if $x_i\in [\frac{t_i}{n}, \frac{t_i+1}{n})$ for each $i\in [2]$ and $\{v_{t_1},v_{t_{2}}\} \in E(G)$ and 0 otherwise.
}

For a graph $F$ on $[r]$, define its \emph{homomorphism density} in a graphon $W$ by
$$
t(F,W):=\int_{[0,1]^r} \prod_{ij\in E(F)} W(x_i,x_j)\, \prod_{i=1}^{r} \dd \mu(x_i).
$$
 In particular, $t(K_2,W)$ is called \emph{the (edge)-density} of the graphon $W$.  
 If $W=W_G$, then we get the \emph{homomorphism density} $t(F,G):=t(F,W_G)$ of $F$ in $G$, which is the probability 
 that a uniformly at random chosen function $f: V(F) \rightarrow V(G)$ maps every edge of $F$ to an edge of~$G$.

 A sequence of graphons $W_n$ \emph{converges} to $W$ if for every graph $F$ we have 
 $$
 \lim_{n\to \infty} t(F,W_n)=t(F,W).
 $$
 In the special case when $W_n=W_{G_n}$, we get the \emph{convergence of graphs} $G_n$ to $W$. Let us call
two graphons $U$ and $W$ \emph{weakly isomorphic} if $t(F,U)=t(F,W)$ for every graph~$F$. {Theorem~13.10} in \cite{Lovasz:lngl} gives several equivalent definitions of weak isomorphism. Let $[W]$ denote the equivalence class of a graphon~$W$ up to weak isomorphism and let 
$$\mathcal{W}:=\{\,[W]:\mbox{graphon }W\,\}.$$

If we fix some enumeration $\mathcal{F}=\{F_1,F_2,\dots\}$ of all graphs up to isomorphism, then one can identify each $[W]\in\mathcal{W}$ with the sequence $(t(F,W))_{F\in\mathcal{F}}\in [0,1]^{\mathcal{F}}$ and the above convergence is the one corresponding to the product topology on $[0,1]^{\mathcal{F}}$. Since the product of compact sets is compact, the closed subspace $\mathcal{W}\subseteq [0,1]^{\mathcal{F}}$ is compact. As $\mathcal{F}$ is countable,  every infinite sequence of graphons/graphs has a convergence subsequence. Also, for each $F$, 
the function $W\mapsto t(F,W)$ is continuous (as it is just the projection on the $F$-th coordinate). Another key property of graphons is that finite graphs are dense in $\C W$.

Let $B \subseteq [0,1]$ be a Borel subset of $[0,1]$ with $\mu(B) >0$. We
define the \emph{graphon $W[B]$ induced by $B$} to be the graphon $(W,\mu')$ with 
$\mu'(A):=\mu(A\cap B)/\mu(B)$ for Borel $A\subseteq [0,1]$.  
 Since the new measure $\mu'$ is 0 on $[0,1]\setminus B$, this effectively restricts everything to $B$ (while the scaling ensures that $\mu'$ is a probability measure).




\begin{definition}[Graphon family~$\C H_r$] \label{def: HrGk} Let $r\ge 3$.
Define
$\mathcal{F}:= \{\, [W] \in \mathcal{W} : t(K_r,W)=0  \,\}$ and, 
for each integer $k\ge r-1$, let 
$\mathcal{G}_{k}$ be the set of weak isomorphism classes of graphons $(W,\mu)$ satisfying the following:
there exist a real $c\in (\frac{1}{k+1}, \frac{1}{k}]$ and a Borel partition $[0,1] =\Omega_1\cup \dots \cup \Omega_{k}$ with $\mu(\Omega_i)=c$ for each $i\in [k-1]$ such that 
\begin{enumerate}
\item $t( K_2, W[\Omega_{k}] ) = 2c(b-c)/b^2$, where $b:=1-(k-1)c$,
\item $t(K_3,W[\Omega_{k}])=0$,
\item For all  $x_1 \in \Omega_{k_1}$ and $x_2\in \Omega_{k_2}$, we have 
$W(x_1,x_2)=  \left\{ \begin{array}{ll}
0, & \text{ if } k_1=k_2\leq k-1 , \\
1, & \text{ if } k_1,k_2\in [k]\text{ satisfy }k_1\neq k_2.
\end{array}\right.$
\end{enumerate}
Further let
$$\mathcal{H}_{r} := \{\,[\mathbf{1}_{[0,1]^2}]\,\} \cup \mathcal{F} \cup \bigcup_{k=r-1}^{\infty} \mathcal{G}_{k}.$$
\end{definition}

This is a direct analogue of the graph family $\C H_{r,n}$: for example, to constract a graphon in $\C G_k$ we take a ``complete partite'' graphon with parts $\Omega_1,\dots,\Omega_{k}$
and add a triangle-free graphon into the last part.
 In fact, by Lemma~\ref{lm:ConvHr}, $\cH_r$ is precisely the set of possible limits of increasing graph sequences $(G_n)_{n=1}^\infty$ with each $G_n\in \cH_{r,n}$. 
 This (or an easy direct calculation) gives that $t(K_r,W)=h_r(t(K_2,W))$ for each 
$W\in \mathcal{H}_r$, where $h_r$ is defined in~\eqref{eq:hr}.
 Also, note that, by definition, $\cH_3\subsetneq \cH_4\subsetneq 
\ldots\subseteq \mathcal W$.  

The result of Reiher~\cite{Reiher16} that determines  the function $g_r$ can be equivalently rephrased in the language of graphons as follows.

\begin{theorem}[\cite{Reiher16}]\label{thm: clique density}
For each integer $r\geq 3$ and a graphon $W$, it holds that $t(K_r,W) \geq h_r(t(K_2,W))$.
\end{theorem}

We say that a graphon $W$ is \emph{$K_r$-extremal} if $t(K_r,W) = g_r(t(K_2,W))$. In other words, it is $K_r$-extremal if it has the minimum $K_r$-density among all graphons with the same edge density.
In fact, the asymptotic structure result for triangles by Pikhurko and Razborov was first derived via a statement about graph limits (see~\cite[Theorem~2.1]{PikhurkoRazborov17}):

\begin{theorem}[\cite{PikhurkoRazborov17}]\label{thm: K3 stability}
A graphon $W$ is $K_3$-extremal if and only if $[W]\in \mathcal{H}_3$.
\end{theorem}

 Our main result in terms of graphons is as follows.

\begin{theorem}\label{thm: main} 
For each $r\geq 3$, a graphon $W$ is $K_r$-extremal if and only if  $[W]\in \mathcal{H}_r$.
\end{theorem}

This completely characterises all graphons achieving the equality in Theorem~\ref{thm: clique density} (and implies Theorem~\ref{th:graph}, see Section~\ref{sec-graphon-to-graph}).

\section{Preliminaries}

Let $\N:=\{1,2,\dots\}$ be the set of natural numbers. For $t\in \I N$, let $\mathbf{I}_t:= [1-\frac{1}{t}, 1-\frac{1}{t+1})$; these intervals partition $[0,1)$. Also, we denote $[t]:=\{1,\dots,t\}$ and, for a set $X$, let ${X\choose t}$ consist of all $t$-subsets of~$X$.
For $k,t\in\I N$, the \emph{$t$-th falling power of $k$} is $k^{(t)}:=k(k-1)\dots (k-t+1)$; note that
$k^{(t)}=0$ if $t>k$. 
 The \emph{indicator function} $\mathbf{1}_Y$ of a set $Y$
assumes value 1 on $Y$ and is 0  otherwise.
An unordered pair $\{x,y\}$ may be abbreviated to~$xy$.

A \emph{graph} is always a finite graph with non-empty vertex set. For a graph $G=(V,E)$ and bijection $\phi$ from $V$ to some set $X$, we denote by $\phi(G)$ the graph on $X$ with edge set $\phi(E):=\{\,\{\phi(x),\phi(y)\}\mid \{x,y\}\in E(G)\}$.  The \emph{edit distance} $|G\triangle H|$ between two graphs $G$ and $H$ of the same order is the minimum of $|E(G)\triangle E(\phi(H))|$ over all bijections $\phi: V(H)\rightarrow V(G)$; in other words, this is the minimum number of edge edits needed to make $G$ and $H$ isomorphic.

For $a,b,c\in \mathbb{R}$ we write $a = b\pm c$ if $b-c \leq a \leq b+c$. 
The constants in the hierarchies used to state our results have to be chosen from right to left. More precisely, if we claim that a result holds whenever e.g.\ $c \ll b \ll a_1,\dots, a_s$,
then this means that there are coordinate-wise non-decreasing functions $f:(0,1] \rightarrow (0,1]$ and $g:(0,1]^{s} \rightarrow (0,1]$ such that the claimed result holds whenever $0<c < f(b)$ and $0<b< g(a_1,\dots, a_s)$.

\subsection{Some notions and results of measure theory}\label{MT} 
Let us recall some basic notions that apply when $\C A$ is a $\sigma$-algebra on a set $X$ and $\nu$ is a measure on $(X,\C A)$. A function $f:X\to\I R$ is called \emph{$\C A$-measurable} if
the preimage of any open (equivalently, Borel) subset of $\I R$ is in~$\C A$. This class of functions is closed under arithmetic operations, pointwise limits, etc., see e.g.~\cite[Section~2.1]{Cohn13mt}.
A set $Y\subseteq X$ is called \emph{($\nu$-)\,null} if there is $Z\in\C A$ with  $Z\supseteq Y$ and $\nu(Z)=0$. We say that a property holds \emph{($\nu$-)\,a.e.} if the set of $x$ where it fails is $\nu$-null.
The \emph{$\nu$-completion} $\C A_\nu$ of $\C A$ consists of those $A\subseteq X$ for which there exist $B,C\in\C A$ with $B\subseteq A\subseteq C$
and $\nu(B)=\nu(C)$; equivalently, $\C A_\nu$ 
is the $\sigma$-algebra generated by the union of $\C A$ and all $\nu$-null sets.

For $k\in\I N$, let $\C B([0,1]^k)$ consist of all \emph{Borel} subsets of $[0,1]^k$ (i.e., it is the $\sigma$-algebra generated by open subsets of $[0,1]^k$).
It is easy to show (see e.g.~\cite[Example~5.1.1]{Cohn13mt}) that $\C B([0,1]^k)$
is equal to the product 
of $k$ copies of the Borel $\sigma$-algebra on~$[0,1]$. When $k$
is clear, we abbreviate $\C B([0,1]^k)$ to~$\C B$.

Let $\mu$ be a probability measure on $([0,1],\C B)$. By $\mu^k$, we denote the measure on~$([0,1]^k,\C B)$ which is the product of $k$ copies of~$\mu$. We call the sets in the $\mu^k$-completion of $\C B([0,1)^k)$ \emph{measurable} and, when $k$ is understood, denote this $\sigma$-algebra by~$\C B_\mu$.
We call a function $f: [0,1]^k \to\I R$ \emph{Borel} (resp.\ \emph{%
measurable}) if it is $\C B$-measurable (resp.\ $\C B_\mu$-measurable).%

Let us state some results that will be useful for us. The first one is an easy consequence of the
countable addivitity of a measure, see e.g.~\cite[Proposition~1.2.5]{Cohn13mt}.

\begin{lemma}[Continuity of measure]
\label{lm:CM} 
For every measure space $(X,\C A,\nu)$
and every nested sequence $X_0\subseteq X_1\subseteq X_2\subseteq\ldots$ of sets in $\C A$, 
the measure of their union $\cup_{n\in\I N} X_n$ is equal to $\lim_{n\to\infty} \nu(X_n)$.
\end{lemma}

The following result will allow us to work with just Borel sets and functions.

\begin{lemma}\label{lem:clean} For every measure space $(X,\mathcal{A},\nu)$ and an $\mathcal{A}_\nu$-measurable function $f:X\to [0,1]$, there is an $\C A$-measurable function $f':X\to[0,1]$ such that $f'=f$ $\nu$-a.e.
\end{lemma}
\begin{proof}
\hide{
Recall that the completion $\mathcal{A}_\nu$ consists of exactly those sets $A$ such that for some $B,N\in\mathcal{A}$ with $\nu(N)=0$ we have $A\bigtriangleup B\subseteq N$. Also, a function $f':X\to[0,1]$ is $\mathcal{A}$-measurable if for every Borel set $B\subseteq [0,1]$ its preimage  $\{x\in X\mid f'(x)\in B\}$ is in $\mathcal{A}$.
}
For $x\in X$, write $f(x)\in [0,1]$ in binary, $f(x)=\sum_{i=0}^\infty b_i(x)2^{-i}$ where each $b_i(x)\in\{0,1\}$ and, for definiteness, we require that infinitely many of $b_i(x)$ are 0 (thus we do not allow expansions where all digits are eventually 1). Note that each $B_i:=\{x\in X\mid b_i(x)=1\}$ is $\mathcal{A}_\nu$-measurable because it is the preimage under the measurable function $f$ of the set consisting of $r\in [0,1]$ such that $i$-th binary digit of $r$ is 1, which is Borel as a union of some intervals.

By the definition of $\mathcal{A}_\nu$, for each $i\in\I N$ there are $A_i,N_i\in\mathcal{A}$ such that $\nu(N_i)=0$
and $A_i\bigtriangleup B_i\subseteq N_i$. Let $g:=\sum_{i=0}^\infty 2^{-i}\mathbf{1}_{A_i}$. Then $g$ is $\C A$-measurable as the countable convergent sum of $\C A$-measurable functions. Also, the set where $f$ and $g$ differ is a subset of $N:=\cup_{i=0}^\infty N_i$, which has measure 0.
Some points in $A_0 \subseteq N_0$ may have $g$-value greater than $1$, so let $f'(x):=\min\{g(x),1\}$. The set where the $\C A$-measurable function $f'$ differs from $f$ is still a subset of the null set $N$, as required.
\end{proof}

The following result will be frequently used (allowing us, in particular, to change the order of integration), so we state it fully. For a proof, see e.g.~\cite[Proposition 5.2.1]{Cohn13mt}. 
	
	\begin{theorem}[Tonelli's theorem]
	\label{th:Tonelli}
	Let $(X,\mathcal{A},\mu)$ and $(Y,\mathcal{C},\nu)$ be $\sigma$-finite measure spaces and let $f:X\times Y\to [0,\infty]$ be $\mathcal{A}\times\mathcal{C}$-measurable. Then
	\begin{enumerate}
		\item for every $x\in X$ the function $f(x,\cdot): Y\to [0,\infty]$, $y\mapsto f(x,y)$, is $\mathcal{C}$-measurable;
		\item the function $x\mapsto \int_Y f(x,y)\dd\nu(y)$ is $\mathcal{A}$-measurable;
		\item we have $\int_X \left( \int_Y f(x,y)\dd\nu(y)\right)\dd\mu(x)=\int_{X\times Y} f(x,y)\dd(\mu\times \nu)(x,y)$.
		\end{enumerate}	
	\end{theorem}

Also,  we will need the following result.

\begin{theorem}[Sierpi\'nski's theorem]
\label{thm: nonatomic} If $(X,\mathcal{A},\nu)$ is a non-atomic measure space with $\nu(X)<\infty$, then for every $\rho\in [0,\nu(X)]$ there is $Y\in \mathcal{A}$ with $\nu(Y)=\rho$.  \end{theorem}

\begin{proof} We will use only the special case  $(X,\mathcal{A})=([0,1],\C B)$ of the theorem, whose proof is very simple. Namely,
the function $x\mapsto \mu([0,x))$ for $x\in [0,1]$ is continuous by Lemma~\ref{lm:CM} (and by the non-atomicity of $\mu$); now, the Intermediate Value Theorem gives the required. 
\end{proof}


\subsection{Further results on  graphons}\label{graphons} 
Usually, a graphon $W$ is defined as a symmetric measurable function $[0,1]^2\to [0,1]$ with respect to the Lebesgue measure~$\mu$ on~$[0,1]$. 
Given such $W$, we can, by  Lemma~\ref{lem:clean}, modify it on a $\mu^2$-null set to obtain a Borel function $U':[0,1]^2\to [0,1]$. Then letting $U(x,y):=\frac12(U'(x,y)+U'(y,x))$ for $(x,y)\in [0,1]^2$, we obtain a Borel symmetric function which is still a.e.\ equal to~$W$. Thus $[U]=[W]$ and our requirement that graphons are represented by Borel functions is not a restriction. Also, on the other hand, we did not enlarge $\C W$ by allowing non-uniform measures on $[0,1]$ as our definition is a special case of the general form of graphons, as defined in~\cite[Section~13.1]{Lovasz:lngl}. The motivation for our definition  is that our proof requires changing the measure a few times while the assumption that the function $W$ is Borel ensures that all sets and functions that we will encounter  are everywhere defined and Borel.

We could have also applied the so-called \emph{purification} of the graphon introduced by Lov\'asz and Szegedy~\cite{LovaszSzegedy10misc}
(see also~\cite[Section~13.3]{Lovasz:lngl}) which would eliminate a few (simple) applications of the continuity of measure in our proof.
However, we decided against using this (non-trivial) result as this could obscure the simplicity of this step. 


One consequence of Tonelli's theorem (Theorem~\ref{th:Tonelli}) and the identity $\C B([0,1]^2)=\C B([0,1])\times \C B([0,1])$
(and our requirement that the function $W$ is Borel) is that for every $x\in[0,1]$ the \emph{section} $W(x,\cdot):[0,1]\to[0,1]$, $y\mapsto W(x,y)$, is a Borel function. 

For a graph $F$ on $[r]$ with $1,\dots, k$ designated as roots, define its \emph{rooted homomorphism density} in $W$ by
 $$
 t_{x_1,\dots, x_k}(F,W):=\int_{[0,1]^{r-k}} \prod_{ij\in E(F)} W(x_i,x_j) \prod_{i=k+1}^{r} \dd \mu(x_{i}),\quad
  \mbox{for } (x_1,\dots, x_k) \in[0,1]^k.
 $$
 By Tonelli's theorem, this is an (everywhere defined) Borel function $[0,1]^k\to [0,1]$.

Note that if $F'$ is obtained from a graph $F$ by rooting it on a fixed vertex, then
$$\int_{[0,1]} t_x(F',W) \dd \mu(x) = t(F,W).$$ 
If $F$ has two roots 1 and 2 that are adjacent and $F^-$ is obtained from $F$ by removing the edge $\{1,2\}$, then $t_{x_1,x_2}(F,W)=W(x_1,x_2)\,t_{x_1,x_2}(F^-,W)$.

For example, if $W=W_G$ for a graph $G$ with $V(G)=[n]$ and $v_1\in [n]$ then, for all $x\in [\frac{v_1-1}n,\frac{v_1}n)$, $t_{x}(F,G)$ is the \emph{rooted density} of $F$ in $(G,v_1)$, namely, the probability for independent uniformly distributed vertices $v_2,\dots,v_r\in V(G)$ that for every $ij\in E(F)$ we have $v_iv_j\in E(G)$. 

Define the \emph{degree} of $x\in [0,1]$  in $W$ by 
 $$d_W(x):=t_x(K_2,W)=\int_{[0,1]} W(x,y)\dd\mu(y).$$ 
 When the graphon is understood, we may abbreviate $d_W(x)$ to $d(x)$. By Tonelli's theorem, $d(x)$ is defined for every $x\in[0,1]$ and 
 \begin{align}\label{eq: degree tK2}
 \int_{[0,1]} d_W(x)\dd\mu(x)=t(K_2,W).
 \end{align}

 \begin{definition}\label{def: nbrhd}
 	For a graphon $(W,\mu)$ and $x\in [0,1]$ with $d(x)\not=0$, define the \emph{neighbourhood} $N_W(x)=N(x)$ of $x$ in $W$ as the graphon $(W,\mu')$, where 
 	$$\dd\mu'(y):=\frac{W(x,y)}{d(x)}\dd\mu(y),$$
 that is, $\mu'(A):=\int_A \frac{W(x,y)}{d(x)}\dd\mu(y)$ for Borel $A\subseteq [0,1]$.
 \end{definition}
 
 Note that, in the above definition, the function $W$ remains the same and only the measure is changed.  With this definition, we have that for every $r\ge 2$
  \begin{eqnarray}\label{eq: nhd relation}
  t(K_r,N_W(x))&=&\int_{[0,1]^r} \prod_{1\le i<j\le r} W(x_i,x_j) \prod_{i=1}^{r} \dd \mu'(x_{i})\nonumber \\
   &=& \int_{[0,1]^r} \prod_{1\le i<j\le r} W(x_i,x_j) \prod_{i=1}^r \frac{W(x,x_i)}{d_W(x)}\ \prod_{i=1}^{r} \dd \mu(x_{i})\nonumber\\
      &=& \frac{t_x(K_{r+1},W)}{(d_W(x))^r}.
 \end{eqnarray}



For $r, \ell,m \in \mathbb{N}$,  $B\subseteq [0,1]^{m}$ and $\mathbf{y} =(y_1,y_2) \in [0,1]^2$, define
\begin{align}\label{eq: B def B}
\begin{split}
B^{r,\ell} &:=\textstyle \left\{ \mathbf{x} \in [0,1]^{r}: \big|\big\{ S \in \binom{[r]}{m}: (x_{i})_{i\in S} \in B \big\}\big| = \ell \right\}, \\
B^{r,\ell}_{+} &:= \bigcup_{i\geq \ell} B^{r,i}, \\
 B^{r,\ell}(\mathbf{y})&:= \Big\{ (y_3,\dots, y_r) \in [0,1]^{r-2} : (y_1,\dots, y_r) \in B^{r,\ell}\Big\}, \text{ and } \\
B^{r,\ell}_{+} (\mathbf{y} )&:= \bigcup_{i\geq \ell} B^{r,i}(\mathbf{y}).
\end{split}
\end{align}
For example, $B^{r,\ell}$ consists of those $r$-tuples from $[0,1]$ such that the number of  $m$-subtuples that belong to $B\subseteq[0,1]^m$ is exactly~$\ell$.  We will later need the property that $B^{r,2}_+$ is much smaller  in measure than $B \subseteq[0,1]^2$. This is not true in general: for instance, consider $U\subseteq [0,1]$ with $\mu(U)=\eta\ll 1/r$ and $B=\{(x,y) : \{x,y\} \cap U \neq \emptyset \}$ when we have $\mu^r(B^{r,2}_{+}) = \Omega(\mu^2(B))$.
However, the following lemma gives the desired property provided that we can pass to a subset of $B$ first. We call a set $B\subseteq [0,1]^2$ \emph{symmetric} if $(x,y)\in B$ implies that $(y,x)\in B$.

\begin{lemma}\label{lem: shrinking}
Let $\mu$ be a non-atomic measure on $([0,1],\C B)$,  $\eta^{-1}\in \mathbb{N}$
and $B \subseteq [0,1]^2$ be a symmetric Borel set.
Then there exists a symmetric Borel subset $C \subseteq B$ satisfying the following for all $r\geq 3$:
\begin{enumerate}[label={\rm (C\arabic*)}]
\item \label{B'1} $\eta^2 \mu^2(B)\leq\mu^2(C) \leq \mu^2(B)$;
\item \label{B'2} For each $\mathbf{x} \in [0,1]^2$,
we have $\mu^{r-2}\left( C^{r,2}_{+}(\mathbf{x}) \right) \leq 2r\eta$;
\item \label{B'3} $\mu^r( C^{r,2}_{+} ) \leq r^3 \eta \mu^2(C)$.
\end{enumerate}
\end{lemma}
\begin{proof}
Let $t:=\eta^{-1}$. Using Theorem~\ref{thm: nonatomic}, 
we can partition $[0,1]$ into Borel sets $I_1,\dots, I_{t}$ with $\mu(I_i) = \eta$ for each $i\in [t]$. For all $i,j\in [t]$, let $I_{i,j}:= I_i\times I_j$. 
As $\mu^2(B) = \sum_{i,j\in [t]} \mu^2(B\cap I_{i,j})$, there exists $(i_0,j_0)\in [t]^2$ such that $\mu^2(B\cap I_{i_0,j_0}) \geq \eta^2 \mu^2(B)$.
Let $$C:= (B\cap I_{i_0,j_0}) \cup (B\cap I_{j_0, i_0}).$$

Clearly, $C$ is a symmetric Borel set that satisfies~\ref{B'1}.

Given $\mathbf{x} =(x_1,x_2)\in [0,1]^2$, we consider the following random experiment. 
We choose $x_3,\dots, x_r \in [0,1]$ independently at random with respect to the probability measure $\mu$.
Let $E$ be the event that $ |\{ ij\in \binom{[r]}{2} : (x_i,x_j) \in C\}| \geq 2 $. Note that by the definition of $C$, if the event $E$ happens then at least one of $x_3,\dots, x_r$ lies inside $I_{i_0}\cup I_{j_0}$.
Thus the probability of $E$ satisfies
\begin{eqnarray*}
\mathbb{P}[E]  
&\leq& \sum_{i= 3}^{r} \mathbb{P}\left[ x_i \in I_{i_0}\cup I_{i_{j_0}} \right] \leq 2r\eta.
\end{eqnarray*}
This implies \ref{B'2}. 

Finally, by the symmetry between the variables $x_i$, \ref{B'2} implies that
\begin{eqnarray*}\mu^r( C^{r,2}_{+} ) &\leq& \binom{r}{2} \int_{\mathbf{x} \in C} \mu^{r-2}\left( C^{r,2}_+(\mathbf{x}) \right) \prod_{i\in [2]} \dd \mu(x_{i}) 
	\le {r\choose 2} \mu^2(C)\cdot 2r\eta \leq  r^3 \eta \mu^2(C).
\end{eqnarray*}
Thus we have \ref{B'3}.
\end{proof}

\subsection{Properties of $\mathcal{H}_r$ and $h_r$}
\label{subsec: hr}

Everywhere in this section, $r\ge 3$ is fixed. In order to deal with the graphon family $\mathcal{H}_r$, it is convenient to define some related parameters. 

For $t,\ell\in\I N$ with $\ell\ge 2$ and $\gamma\in\mathbb{R}$,
define 
 \begin{align}
\kappa_{\ell,t}(\gamma)& := \ell!\,\left({t\choose \ell} \gamma^\ell+{t\choose \ell-1}\gamma^{\ell-1}(1- t\gamma)\right)\nonumber\\
 &= t^{(\ell-1)}\gamma^{\ell-1}(\ell-(\ell-1)(t+1)\gamma). \label{eq:kappa}
\end{align}

 Note that if $0\le \gamma\le \frac{1}{t}$ then  $\kappa_{\ell,t}(\gamma)$ is the asymptotic density of $\ell$-cliques in a complete partite graph with
 $t$ parts of size $\gamma n$ and one part of size $(1- t\gamma)n$ as $n\to\infty$. 
 Next, for $x\le 1-\frac{1}{t+1}$, let
\begin{align}\label{eq:gamma}
\gamma_{t}(x):=\frac{1}{t+1} + \frac{\sqrt{t(t-(t+1)x)}}{t(t+1)}.
\end{align}
 This formula comes from taking  $\gamma_{t}(x)$ to be the larger root $\gamma$ of 
the quadratic equation 
$x=\kappa_{2,t}(\gamma)$.
 Further, again for $x\le 1-\frac1{t+1}$, define
\begin{align}\nonumber
p_{r,t}(x) 
&:= \kappa_{r,t}(\gamma_{t}(x))\\
&=\ \frac{t^{(r-1)}}{t^{r}(t+1)^{r-1}}\left(t+ \sqrt{t(t -(t+1)x)}\right)^{r-1} \left(t - (r-1)\sqrt{t(t- (t+1)x)}\right).\label{eq:prt}
\end{align}

With this preparation we are ready to define the two main parameters, $k$ and $c$, associated with the edge density~$\alpha\in [0,1)$, namely
\begin{align}
\begin{split}
k&\textstyle =k(\alpha)\in\I N ~ \text{ such that }~ \alpha\in \bI_{k} \mbox{ (that is, $1-\frac{1}{k}\le \alpha<1-\frac{1}{k+1}$)},\\
c&=c(\alpha):=\textstyle \gamma_{k}(\alpha). \label{def: k}
\end{split}
\end{align}
Note that $c=c(\alpha)$ in~\eqref{def: k} is the same as in~\eqref{eq:hr} and
Definition~\ref{def: HrGk}; also,~\eqref{eq:gamma} gives an explicit formula for~$c$. 

In other words, the function $c:[0,1)\to (0,1]$ is obtained by taking $\gamma_{t}$ on the interval $\bI_t$
for $t\in\I N$. (Recall that these intervals partition~$[0,1)$.) Since  the left and right
limits of the function $c$ coincide at any internal boundary point $1-\frac{1}{t+1}$, with $t\in\I N$ (namely, both are $\frac1{t+1}$), the explicit formula in~\eqref{eq:gamma} gives that $c$ is a continuous and strictly monotone decreasing function on $[0,1)$.

Also, easy calculations based on e.g.~\eqref{eq:gamma} show that for all $\alpha\in [0,1)$ we have, with $c=c(\alpha)$ and $k=k(\alpha)$, that
 \begin{equation}\label{eq:c}
 \frac1{k+1}<c\le\frac1{k}.
 \end{equation} 
Furthermore,  $h_r(\alpha)=\kappa_{r,k}(c)=p_{r,k}(\alpha)$, where the function $h_r(\alpha)$ was defined in~\eqref{eq:hr}. 
In fact, as we will show in Lemma~\ref{lem: linear extension}, $h_r(\alpha)=\max\{p_{r,t}(\alpha)\mid t\ge k\}$, that is, $h_r$ is the maximum of those functions $p_{r,t}$ that are defined at a given point, with $p_{r,t}$ being a largest one on~$\bI_{t}$.

The following lemma computes the first two derivatives of $p_{r,t}$ (and thus of $h_r$ in all interior points of each $\mathbf{I}_t$), where we write these derivatives in terms of $\gamma_{t}$ for convenience. Note that $h_r$ is not differentiable at points $1-\frac{1}{t}$ for integers $t\ge r-1$: the left and right derivatives of $h_r$ exist at these points but are different.

\begin{lemma}\label{eq: h'}
For each $t\in\I N$ and $x\le 1-\frac{1}{t+1}$, we have that
$$
p_{r,t}'(x)
= \binom{r}{2}  (t-1)^{(r-2)}\gamma_{t}(x)^{r-2} ~~ \enspace \text{ and }~~ \enspace  
p_{r,t}''(x)= \frac{ 3\binom{r}{3}  (t-1)^{(r-2)} \gamma_{t}(x)^{r-3}}{2 t(1-(t+1)\gamma_{t}(x))}.$$
\end{lemma}

\begin{proof}
It is easy to calculate that
$\kappa_{2,t}'(\gamma)=2t(1-(t+1)\gamma)$. Denote $\gamma=\gamma_{t}(x)$. Since $x=\kappa_{2,t}(\gamma)$, we  derive from \eqref{eq:kappa} by the Implicit Function Theorem 
that
\begin{eqnarray*}
p_{r,t}'(x)
= \frac{\kappa_{r,t}'(\gamma)}{\kappa_{2,t}'(\gamma)}
=
\frac{t^{(r-1)} r(r-1) \gamma^{r-2}(1-(t+1)\gamma)}{2t(1-(t+1)\gamma)}
\ =\  \binom{r}{2}  (t-1)^{(r-2)} \gamma^{r-2}.
\end{eqnarray*}
Likewise, $p_{r,t}''(x)=\frac{\dd}{\dd \gamma}\left(\binom{r}{2}  (t-1)^{(r-2)}\gamma^{r-2}\right)/\kappa_{2,t}'(\gamma)$, giving the stated formula.
\end{proof}

An informal  explanation of the above formula for $h_r'(\alpha)$ is that this derivative measures the increament in the number of $r$-cliques in $H_{\alpha,n}$, normalised by ${n\choose 2}/{n\choose r}$, when we increase $\alpha$ as $n\to\infty$. When we add $\lambda=o(n^2)$ new edges between the last two parts, we create around $\lambda {k-1\choose r-2}(cn)^{r-2}$ copies of $K_r$ while the change in the ratio $c$ has negligible effect because the (optimal) vector of part ratios is critical. Now note that ${k-1\choose r-2}(cn)^{r-2}\cdot{n\choose 2}/{n\choose r}=\binom{r}{2}  (k-1)^{(r-2)}c^{r-2}+o(1)$.

The following lemma directly follows from the previous lemma and Taylor's approximation. 

\begin{lemma}\label{lem: Taylor} Let $\alpha\in [0,1)$. Let
$k\in \mathbb{N}$ and $c$ be as in \eqref{def: k}. If $\alpha\not=1-\frac{1}{k}$
(that is, $\alpha$ is in the interior of $\bI_k$), then there is $\epsilon>0$ such that for each $\alpha' =\alpha\pm \epsilon$, we have
$\alpha'\in \mathbf{I}_k$ and
$$h_r(\alpha') = h_r(\alpha)+ \binom{r}{2}  (k-1)^{(r-2)}c^{r-2} (\alpha'-\alpha) \pm |\alpha'-\alpha|^{3/2}.\qed$$
\end{lemma}

\hide{
We remark that in the above lemma, as $\epsilon$ is sufficiently small compared to $\alpha$ and $1/k$ ($\epsilon \ll \alpha,1/k$), $\epsilon < \min\{ \alpha-(1-\frac{1}{k+1}), 1-\frac{1}{k+2} -\alpha\}$, hence $\alpha'\in \mathbf{I}_k$. 
}

The following lemma proves Theorem~\ref{thm: main} for the special edge densities where the function $h_r$ is not differentiable.
\hide{
\begin{lemma}[\cite{LovaszSimonovits83}]\label{lem: LS stability}
Suppose $0< 1/n \ll \eta \ll \epsilon \leq 1$.
If an $n$-vertex graph $G$ satisfies 
$$t(K_2,G) =  \left(1-\frac{1}{t}\right) \pm \eta \quad \text{and} \quad
t(K_r, G) = h_r\left(1- \frac{1}{t}\right) \pm \eta,$$
then we have $|G \triangle T_t(n)| < \epsilon n^2.$
\end{lemma}
}

\begin{lemma}\label{lem: at boundary}
Let $t\ge r-1$ be integer, $\alpha=1-\frac1t$, and let $W$ be a graphon with $t(K_2,W) = \alpha$. If $ t(K_r,W) = h_r(\alpha)$, then $W\in [W_{K_t}]$. 
\end{lemma}
\begin{proof}
The quickest way to prove the lemma is to use the weaker version of a result of Lov\'asz and Simonovits~\cite[Theorem~2]{LovaszSimonovits83} that every graph of order $n\to\infty$ with $(\alpha+o(1)){n\choose 2}$ edges and $(h_r(\alpha)+o(1)){n\choose r}$ copies of $K_r$ is $o(n^2)$-close in edit distance to the Tur\'an graph~$T_t(n)$. Applying this result to a sequence of graphs $(G_n)_{n=1}^\infty$, where  $G_n$ has $n$ vertices, that converges to the graphon $W$, we can transform each $G_n$ into $T_t(n)$ by changing $o(n^2)$ adjacencies. This change does not affect the convergence to $W$. Now, the limit of the $t$-partite Tur\'an graphs is clearly $[W_{K_t}]$, giving the required.\end{proof}

\begin{remark}\rm Alternatively, one can prove Lemma~\ref{lem: at boundary} operating  with graphons only. Namely, the proof of Lov\'asz and Simonovits~\cite[Theorems~1--2]{LovaszSimonovits83} for graphons would be to write $t(K_r,W)/t(K_2,W)$ as a telescopic product  over $3\le s\le r$ of $t(K_s,W)/t(K_{s-1},W)$ and bound each ratio separately, using the Cauchy-Schwartz Inequality (with double counting replaced by Tonelli's theorem).
In particular, since $t(K_r,W)$ is smallest possible, the graphon $W$ also minimises the triangle density for the given edge density $\alpha=1-\frac1t$. By unfolding the corresponding argument from~\cite{LovaszSimonovits83}, one can show that the induced density of 3-sets spanning exactly one edge is 0. It follows with a bit of work that, similarly to graphs, $W$ is a complete partite graphon a.e. Now, a routine optimisation (see e.g.~\cite[Theorem~1.3]{Nikiforov11}) shows that, apart a null-set, there are exactly $t$ parts of equal measure.%
	\end{remark}
\hide{Let $\epsilon_1, \dots$ and
$\eta_1, \dots $ be decreasing sequences converging to $0$
and  $n_1, \dots$ be an increasing sequence satisfying $1/n_i \ll \eta_i \ll \epsilon_i $ for each $i\in \mathbb{N}$. Consider the sequence $G_1, G_2,\dots$ of $n_i$-vertex graphs $G_i$ with $$t(K_2,G_i) = \left(1-\frac{1}{t}\right)\pm \eta_i \enspace \text{ and }\enspace
t(K_r,G_i) = h_r\left(1- \frac{1}{t}\right) \pm \eta_i.$$
Such a sequence exists, for example, since $1/n_i \ll \eta_i$, the sequence of random $W$-graphs $\mathbb{G}(n_i,W)$ with $i\in\mathbb{N}$, has this property almost surely, see~\cite[Corollary~2.6]{LovaszSzegedy06}). 

By Lemma~\ref{lem: LS stability}, for each $i\in \mathbb{N}$, we have
$|G_i \triangle  T_t(n_i)| \leq \epsilon_i n_i^2$.
Hence, for every graph $F$, we have 
$t(F, G_{i}) = t(F,T_{t}(n_i)) \pm \epsilon_i|F|!$. 
As  $\epsilon_i$ approaches to zero, we have
$$\lim_{i \to\infty} t(F,G_{i})=\lim_{i \to\infty} t(F, T_t(n_i))=t(F,W_{K_t}),$$ as desired.
}

We will also need the following result, which is essentially a consequence of the piecewise concavity of the function $h_r$.

\begin{lemma}\label{lem: linear extension}
For every $t\in\I N$ and $\alpha \in [0,1-\frac{1}{t})$, we have that  $h_r(\alpha) \geq p_{r,t}(\alpha)$.
\end{lemma}
\begin{proof}
Let $x_0:=1-\frac{1}{t}$ and define 
$$L_{r,t}(x):=p_{r,t}\left(x_0\right) +  p'_{r,t}(x_0)\left(x-x_0\right),\quad\mbox{for $x\in\I R$}.
$$ 
In other words, $y= L_{r,t}(x)$ is the line tangent to the curve $y= p_{r,t}(x)$ at $x= x_0$.
Since $\gamma_{t}(x)\ge \frac1t> \frac{1}{t+1}$ for $x\le x_0$ by~\eqref{eq:gamma}, Lemma~\ref{eq:  h'}  gives that the function $p_{r,t}$ has the negative second derivative and is thus  concave. 
We conclude that
$p_{r,t}(x) \leq L_{r,t}(x)$ for all $x\leq  x_0$.

Thus we are done if we show that $h_r(x)\ge L_{r,t}(x)$ for all $x\in[0,x_0]$. Note 
that $h_r(x_0)=L_{r,t}(x_0)$.
Since $h_r$ is a continuos function which is differentiable for every $x\in [0,x_0]$ apart finitely many points, it is enough to show by the Mean Value Theorem that $h_r'(x)\le p'_{r,t}(x_0)$ for each $x\in [0, x_0]$ where $h_r$ is differentiable. So,
let $x\in \bI_s$ with $0\le s< t$. Since $h_r=p_{r,s}$ on $\bI_s$ and, by Lemma~\ref{eq:  h'}, the derivative $p_{r,s}'$ is a decreasing function, it is enough to check that $p_{r,s}'(1-\frac1{s})\le p_{r,t}'(1-\frac1{t})$.
Note that $\gamma_{m}(1-\frac1{m})=\frac1{m}$ for each $m\in\I N$. If $s\ge r-2$, then by Lemma~\ref{eq:  h'} we have that
$$
\frac{p_{r,s}'(1-\frac 1{s})}{p_{r,t}'(1-\frac 1{t})}
=\frac{(s-1)^{(r-2)}(\frac1{s})^{r-2}}{(t-1)^{(r-2)}(\frac1{t})^{r-2}}=\prod_{i=1}^{r-2} \frac{t(s-i)}{s(t-i)} \leq 1
$$
because  $t(s-i)- s(t-i)=i(s-t)\le 0$. If $s\le r-2$, then $p_{r,s}'(1-\frac1{s})=0$ while $p_{r,t}'(x_0)\ge 0$, also giving the desired inequality.
\end{proof}

\begin{lemma}\label{lm:ConvHr}  
Every graphon $(W,\mu)$ in $\C H_r$ is the limit of some sequence $(H_n)_{n=1}^\infty$ where $H_n\in\C H_{r,n}$ for each integer $n\ge1$.
Also, for all integers $n_1<n_2<\ldots$ and graphs $H_i\in \C H_{r,n_i}$ such that the sequence $(H_i)_{i=1}^\infty$ converges, its limit is in $\C H_r$.
\end{lemma}
\begin{proof} Assume that $\alpha:=t(K_2,W)<1$ (as otherwise we can take $H_n$ to be the complete graph) and that $t(K_r,W)>0$ (as $\C H_{r,n}$ contains all $K_r$-free graphs of order~$n$).
Let $\Omega_1\cup\dots \cup \Omega_{k}$ be the partition of the underlying space $[0,1]$ for the graphon $W$, as in Definition~\ref{def:  HrGk}.

For each $n\ge 1$, let $G_n\sim\random nW$ be a graph on $[n]$ which is an \emph{$n$-vertex sample} of $W$, that is, we pick $n$  points $x_{n,1},\dots,x_{n,n}\in [0,1]$ using the probability measure $\mu$ and make $i,j\in [n]$ adjacent with probability $W(x_i,x_j)$, with all choices being independent. Then  the sequence $G_n$ converges to $W$ with probability $1$, see Lov\'asz and Szegedy \cite[Corollary~2.6]{LovaszSzegedy06}. 
Each graph $G_n$ comes with a vertex partition $V_{n,1},\dots,V_{n,k}$, where we put $i\in V(G_n)$ into $V_{n,j}$ if $x_{n,i}\in \Omega_j$. By the Chernoff Bound, we have  that  $|V_{n,j}|/n$ converges to $\mu(\Omega_j)=c$ for every $j\in [k]$ as $n\to\infty$,  with probability $1$.
Since $W$ is an (explicit) $\{0,1\}$-valued function, we know all edges of $G_n$ apart from the ones inside $V_{n,k}$.
Using that $\lim_{n\to\infty}t(K_s,G_n)= t(K_s,W)$ in the special cases $s=2,3$, we derive that $V_{n,k}$ induces $o(n^3)$ triangles in $G_n$ as well as the asymptotically correct number of edges. Fix a sequence $(G_n)_{n=1}^\infty$ that satisfies all above properties.

Now we are ready to show that the edit distance between $G_n$ and some graph in $\C H_{\alpha,n}$ is $o(n^2)$, which will be enough to prove the first part of the lemma. For each $n$, move $o(n)$ vertices between the parts of $G_n$ so that $|V_{n,i}|=\lfloor cn\rfloor$ for each $i\in [k-1]$. (The new adjacencies of a moved vertex are determined by its new part, except if we move a vertex into $V_{n,k}$ we make it adjacent to e.g.\ every other vertex for definiteness.) The new graphs $G_n$ still satisfy the above properties and have the correct part sizes. Using the Triangle Removal Lemma~\cite{RuzsaSzemeredi78,ErdosFranklRodl86} (see
e.g.~\cite[Theorem~2.9]{KomlosSimonovits96}), we make $G_n[V_{n,k}]$ triangle-free by changing $o(n^2)$ adjacencies. The definition of $\C H_{\alpha,n}$ requires to have exactly $\lfloor cn\rfloor\cdot (|V_{n, k}|-\lfloor cn\rfloor)$ edges in~$V_{n,k}$. This can be achieved by~\cite[Lemma~2.2]{PikhurkoRazborov17} which states that if $G$ is triangle-free graph with $m\to\infty$ vertices and $s=e(G)+o(m^2)$ is at most $t_2(m)$, then $G$ is $o(m^2)$-close in edit distance to a triangle-free graph with exactly $s$ edges, as desired.

Let us now show the second part of the lemma. Assume that a sequence $(H_i)_{i=1}^\infty$ contradicts the statement. 
As $\C H_r$ contains the constant-1 graphon, the limiting density  $\alpha:=\lim_{i\to\infty} t(K_2,H_i)$ must be in $[0,1)$. Also, $\lim_{i\to\infty} t(K_r,H_i)>0$ since $\C H_r$ contains all graphons with zero $K_r$-density.

Let $V(H_i)=V_{i,1}\cup \dots\cup V_{i,k-1}\cup U_{i}$ be the partition from the definition of $H_i$. Let $F_i:=H_i[U_i]$. By the compactness of $\C W$, some subsequence of $(F_i)_{i=1}^\infty$ converges to some graphon~$W'$. The limiting graphon $W'$ has zero triangle density.
Since we know all edges of $H_i$ except inside $U_i$, the graphon $W'$ has the correct edge density. Now, define $W\in\C H_r$ 
as in Definition~\ref{def:  HrGk} with $c=c(\alpha)$, $k=k(\alpha)$, and $W[\Omega_{k}]$ being weakly isomorphic to~$W'$. Since we know all adjacencies except inside $\Omega_{k}$, a routine calculation shows that $H_i$ converges to $W$, as required.\end{proof}

\subsection{Asymptotic structure from extremal graphons}\label{sec-graphon-to-graph}

We are ready to show that Theorem~\ref{thm:  main} implies
Theorem~\ref{th:graph} by adopting the analogous step from~\cite[Section~2.2]{PikhurkoRazborov17}.

\begin{proof}[Proof of Theorem~\ref{th:graph}] 
Suppose for the sake of contradiction that Theorem~\ref{th:graph} is false,
which is witnessed by some $r\ge 4$ and $\e>0$. Thus  we can find a sequence $(G_n)_{n\in\I N}$ of graphs of increasing orders $v_n:=v(G_n)$ such that
$t(K_r,G_n)= g_r(t(K_2,G_n))+o(1)$ and each $G_n$ is $\e v_n^2$-far in edit distance from~$\C H_{r,v_n}$.  By using the compactness of $\C W$ and passing to a subsequence, we can additionally assume that
the sequence $(G_n)_{n\in\I N}$ is convergent to some
graphon~$W$. Let $\alpha:=t(K_2,W)$.
Clearly, $t(K_r,W)=g_r(\alpha)$. By Theorem~\ref{thm:  main},
$[W]\in \C H_r$. By Lemma~\ref{lm:ConvHr} pick $H_n\in \C H_{r,v_n}$ such that 
the sequence $(H_n)_{n\in\I N}$ converges to~$W$. 

For two graphs $G$ and $H$ of the same order $n$, define the \emph{cut distance} $\hat \delta_\Box(G,H)$ to be the minimum over all bijections $\phi:V(H)\to V(G)$ of $\hat d(G,\phi(H))$, where for graphs $G$ and $F$ with $V(G)=V(F)$ we define
\begin{equation}\label{eq:cut}
\hat d(G,F):=\max_{S,T\subseteq V(G)} \frac{|\,e_G(S,T)-e_F(S,T)\,|}{v(G)^2},
\end{equation} 
 with $e_G(S,T):=|\{ (x,y)\in S\times T\mid \{x,y\}\in E(G)\}|$.
 Informally speaking, $\hat d(G,F)$ is small if the two graphs on the same vertex set have similar edge distributions over all vertex cuts, while $\hat \delta_\Box$ is the version of $\hat d$
 where we look only at the isomorphism types of the graphs.

Theorems~2.3 and~2.7 in Borgs et al~\cite{BCLSV08} give that $\hat\delta_\Box(G_n,H_n)\to 0$. 
 Namely, \cite[Theorem~2.7]{BCLSV08} states that if two graphs have similar subgraph 
densities, then they are close in the \emph{fractional cut-distance} $\delta_\Box$
(which is defined the same way as $\hat{\delta}_\Box$ except, informally speaking,
 $\phi$ distributes each vertex of $H$ fractionally among $V(G)$), while \cite[Theorem~2.3]{BCLSV08} provides an upper bound of $\hat\delta_\Box$ in terms of $\delta_\Box$.

Up to relabelling of each $H_n$, assume that $\hat d(G_n,H_n)\to 0$. 
Take any $n\in\I N$ and let $v:=v_n$. Fix the partition $V(H_n)=V_1\cup\dots \cup V_{k-1}\cup U$ that was used to define
$H_n\in \C H_{r,v}$. For $i\in [k-1]$, if  we use $S=V_i$ and $T=V_i$ (resp.\ $T=V(H_n)\setminus V_i$) in~\eqref{eq:cut}, then we conclude that $V_i$ spans $o(v^2)$ edges (resp.\ $V_i$ is almost complete to the rest). Thus, by changing $o(v^2)$ adjacencies in $G_n$, we
can assume that the graphs
$G_n$ and $H_n$ coincide except for the subgraphs induced by~$U$.
Suppose that $|U|=\Omega(v)$ for otherwise
we get a contradiction to Lemma~\ref{lem: at boundary}. We have
$$
|e(G_n[U])-e(H_n[U])|= |e(G_n)-e(H_n)|=o(v^2).
$$
Of course, when we modify $o(v^2)$ adjacencies in $G_n$, then the $K_r$-density changes by $o(1)$. Also, each
edge of $G_n[U]$ (and of $H_n[U]$)
is in the same number of $r$-cliques whose remaining $r-2$ vertices are in $V(G_n)\setminus U$. Since $H_n[U]$ is
triangle-free and $G_n$ is asymptotically extremal, we
conclude that $G_n[U]$ spans $o(v^3)$ triangles. We can change $o(v^2)$ adjacencies and make $G_n[U]$ to be
triangle-free by the Triangle Removal
Lemma and have the ``correct'' number of edges by~\cite[Lemma~2.2]{PikhurkoRazborov17}. The obtained graph (which is $o(v^2)$-close in edit distance to $G_n$) is in $\C H_{r,v}$, contradicting our assumption.\end{proof}

\section{Proof of the main result}\label{sec-main-proof}
	
Suppose that Theorem~\ref{thm: main} is not true. Let $r\geq 3$ be the minimum integer such that
there exists a $K_r$-extremal graphon $W= (W,\mu)$ which does not belong to~$\cH_r$. 
Let 
$$
\alpha:= t(K_2,W), \enspace \enspace k:=k(\alpha), \enspace \enspace\text{and} \enspace \enspace c:=c(\alpha) \enspace \text{ as in }\eqref{def: k}.
$$
As $[W]\notin \mathcal{H}_r$, we have $0< \alpha <1$ and Theorem~\ref{thm: K3 stability} implies that $r\geq 4$. 
We may further assume the following properties.
\begin{enumerate}[label={\rm (W\arabic*)}]
\item \label{W2} $t(K_r, W)= h_r(\alpha)>0$ and $t(K_3,W) > h_3(\alpha)$.
\item \label{W3} $\alpha \in \mathbf{I}_k  \setminus\{1- \frac{1}{k} \}$ and $c\in(\frac{1}{k+1}, \frac{1}{k})$.
\end{enumerate}
Indeed, we may assume that $W$ is not $K_3$-extremal as otherwise $W\in \mathcal{H}_3 \subseteq \mathcal{H}_r$ by Theorem~\ref{thm: K3 stability}. This, together with Theorem~\ref{thm: clique density}, implies \ref{W2}.
As $[W_{K_{t}}]  \in \cH_r$ for all $t \in \mathbb{N}$, Lemma~\ref{lem: at boundary} implies \ref{W3}.

Our strategy is as follows. Using \ref{W2}, we will find a point $x\in [0,1]$ such that $t_x(K_r,W)$ is small while $t_x(K_3,W)$ is large. (Recall that these are the densities of respectively $K_r$ and $K_3$ rooted at~$x$.) Note that \eqref{eq: nhd relation} translates these two values together with $d_W(x)$ 
into $t(K_{r-1},N_W(x))$ and $t(K_2, N_W(x))$. Hence, this will eventually enable us to translate the assumption $[W]\notin \mathcal{H}_r$ into some conclusion about $N_W(x)$, which will violate Theorem~\ref{thm: clique density} for~$r-1$.

In order to work in $N_W(x)$, we need to relate $d_W(x)$ and $t_x(K_r,W)$.
For this purpose, we will make use of the following auxiliary functions. For integer $t\ge 3$ and real $x\in [0,1]$, define
\begin{align}\label{eq: def q f}
\begin{split}
q_t(x)&:=
 	(t-1)(d_W(x)-(k-1)c)(k-1)^{(t-2)} c^{t-2}+ (k-1)^{(t-1)}c^{t-1}, \enspace ~~~\text{and }\\
f_t(x)&:= q_t(x) - t_x(K_t,W).
\end{split}
 	\end{align}
 	By
Tonelli's theorem (Theorem~\ref{th:Tonelli}), $d_W(x)$ and $t_x(K_t,W)$ are (everywhere defined) Borel functions of $x\in [0,1]$, so $q_t$ and $f_t$ are also Borel.
Later, in Claim~\ref{cl:1}, we will show that $f_r(x)=0$ for almost all $x$, which provides the desired relation between $d_W(x)$ and $t_x(K_r,W)$.

We first prove the following lemma, which partly motivates the definition of the function~$f_t$. 
\begin{lemma}\label{lem: integral f}
For each integer $t\ge 3$, we have
$$\int_{[0,1]} f_t(x) \dd\mu(x) = h_t( \alpha) - t(K_t, W).$$
\end{lemma}
\begin{proof}
By definition, we have
\begin{eqnarray*}
\int_{[0,1]} f_t(x)\dd\mu(x)
&=&(t-1)\left(\int_{[0,1]} d_W(x) \dd\mu(x)  - (k-1)c\right) (k-1)^{(t-2)}c^{t-2}  \nonumber \\
 & & + (k-1)^{(t-1)} c^{t-1} - \int_{[0,1]} t_x(K_t, W) \dd\mu(x) \nonumber \\
&\stackrel{\eqref{eq: degree tK2},\eqref{def: k}}{=}& (t-1) \big( kc(2-(k+1)c)  - (k-1)c \big) (k-1)^{(t-2)}c^{t-2}  \nonumber \\
 & & + (k-1)^{(t-1)} c^{t-1} - t(K_t,W).
\end{eqnarray*}
Recalling the definition of $h_t(\alpha)$ from \eqref{eq:hr}, one can see that the right hand side above simplifies to $h_t(\alpha) - t(K_t,W)$, as desired.
\end{proof}

We shall try to locate the desired point $x\in [0,1]$ as outlined above in the following subsections.

\subsection{Almost all points are ``$K_r$-typical''}
We further introduce the following two sets. Let
\begin{eqnarray*}
M_0&:=& \{ x\in [0,1]: f_r(x) \neq 0\},
 \end{eqnarray*} 
 that is, $M_0$ is the set of ``$K_r$-atypical'' points. Let
 \begin{eqnarray*}
 N_{0}&:=&\{ x\in [0,1]: f_3(x) < 0 \},
 \end{eqnarray*} 
 that is, $N_0$ is the set of ``$K_3$-heavy'' points.
Note that both sets are  Borel as $f_3$ and $f_r$ are  Borel functions.
We first show that $M_0$ is of negligible measure.
\begin{claim}\label{cl:1} 
We have $\mu(M_0) =0$.
	\end{claim}
\begin{proof}
 The statement that $f_r=0$ a.e.\ follows with some calculations from Razborov's differential calculus~\cite[Corollary~4.6]{Razborov07}.
Informally speaking, the quantity $f_r(x)$  measures the ``contribution" of $x$ to  $h_r(t(K_2,W))-t(K_r,W)$. The terms of $f_r$ that are linear in $d_W(x)$ and $t_x(K_r,W)$ give the gradient when we increase or decrease the density of $\mu$ at $x$ (while the constant term is chosen to make the average of $f_r$ zero).
Here $\alpha=t(K_2,W)$ is in the interior of $\bI_k$, where $h_r$ is differentiable. Since we cannot push $h_r(t(K_2,W))-t(K_r,W)=0$ into positive values by Theorem~\ref{thm: clique density}, it follows that $f_r(x)=0$ for almost every~$x\in[0,1]$.
 
For the reader's convenience, we present a direct proof of the claim. For each $\gamma \geq 0$, let 
$$U_{1,\gamma} := \{x\in [0,1]: f_r(x) >\gamma\} \enspace\text{and}\enspace U_{2,\gamma}:= \{x\in [0,1]: f_r(x)< -\gamma \}.$$
 Note that $U_{1,\gamma}, U_{2,\gamma}$ are both Borel sets for all $\gamma \geq 0$ as the function $f_r$ is Borel. Suppose that $\mu(M_0)= \mu(U_{1,0} \cup U_{2,0})>0$.

By \ref{W2} and Lemma~\ref{lem: integral f}, we have
$\int_{[0,1]} f_r(x) \dd\mu(x)=0$, implying that both $\mu(U_{1,0})$ and $\mu(U_{2,0})$ are positive. Indeed, if say $\mu(U_{1,0})=0$ while $\mu(U_{2,0})>0$, then $\int_{[0,1]} f_r(x)\dd\mu(x)<0$, a contradiction. For $i\in [2]$, as $\bigcup_{\gamma>0} U_{i,\gamma} = U_{i,0}$ has positive measure and $\{ U_{i,\gamma}\}_{\gamma >0}$ forms a nested collection of sets, there exists $\gamma >0$ such that $\mu(U_{i,\gamma})>\gamma$ by Lemma~\ref{lm:CM} (applied to a countable sequence of $\gamma\to0$).

In brief, we derive a contradiction to the minimality of $W$ by replacing a small subset of vertices with negative $f_r$ by those with with positive $f_r$, and showing that this strictly decreases $t(K_r,W)-h_r(t(K_2,W))$. Formally, we choose $\eta, \epsilon$ such that 
$$0< \eta \ll \epsilon  \ll \gamma, 1/k, 1/r, c\leq 1.$$
By Sierpinski's theorem (Theorem~\ref{thm: nonatomic}) and since the measure $\mu$ is non-atomic by
the definition of a graphon,
there exist sets  $U_1 \subseteq U_{1,\gamma}$ and $U_2\subseteq U_{2,\gamma}$ such that
$\mu(U_1) = \mu(U_2) = \eta/2$.  Let $U:= U_1\cup U_2$.
Consider the density function
$$
u(z):=\left\{\begin{array}{ll} 1,& z\in [0,1]\setminus U,\\
1+\epsilon ,& z\in U_1,\\
1-\epsilon ,& z\in U_2.\end{array}\right.
$$
 Then $\int_{[0,1]} u(z)\dd\mu(z)=1$, so $\dd\mu'(z):=u(z)\dd\mu(z)$ is also a Borel probability measure. Let $W'=(W,\mu')$ be the graphon with the same function $W$ but with the new probability measure $\mu'$.
 Recall the definitions in \eqref{eq: B def B}.
As $\mu(U) = \eta \ll 1/r$, the Union Bound gives that 
\begin{eqnarray}\label{eq: Omega 1 measure 1} 
\mu^r( U^{r,1} ) &\leq& r\eta,
\end{eqnarray}
and
\begin{eqnarray}\label{eq: Omega 2 measure} 
\mu^r( U^{r,2}_+ ) &\leq& \sum_{\ell \geq 2} \binom{r}{\ell} \eta^\ell \leq r^2 \eta^2.
\end{eqnarray}
Note that $U^{2,0} = ([0,1] \setminus U)\times([0,1]\setminus U)$.
For each $j\in [2]$, let 
$$V_j:= \left(U_j\times ([0,1]\setminus U)\right) \cup  \left(([0,1]\setminus U) \times U_j\right).$$
Then $U^{2,1} = V_1\cup V_2$.
As $U^{2,0}$ does not contribute to $t(K_2,W')-t(K_2,W)$, we obtain
\begin{eqnarray}\label{eq: K2 diff}
 t(K_2,W') - t(K_2,W) &=& 
\sum_{\ell\in [2]} \int_{U^{2,\ell}} \left( W'(x_1,x_2) \prod_{i\in [2]} \dd\mu'(x_{i}) -   W(x_1,x_2) \prod_{i\in[2]} \dd\mu(x_{i})\right)\nonumber\\
&\stackrel{\eqref{eq: Omega 2 measure} }{=}&  \left(\int_{V_1} -\int_{V_2} \right) \epsilon W(x_1,x_2)   \prod_{i\in [2]} \dd\mu(x_i)
\pm r^2 \eta^2 \nonumber \\
&=& 2\left(\int_{U_1\times [0,1]} -\int_{U_2\times [0,1]} \right) \epsilon  W(x_1,x_2)  \prod_{i\in [2]} \dd\mu(x_i) 
 \pm 2 \mu^2(U^2) \pm  r^2 \eta^2  \nonumber  \\
 &=& 2 \epsilon \left( \int_{U_1} - \int_{U_2} \right)  d_W(x) \dd \mu(x) \pm  2r^2 \eta^2, 
\end{eqnarray}
where the final equality follows from Tonelli's theorem (Theorem~\ref{th:Tonelli}). In particular,
as $\mu(U_1)=\mu(U_2)=\eta/2$ and $d_W(x)\leq 1$ for all $x\in [0,1]$, \eqref{eq: K2 diff} with the fact $\eta \ll \epsilon$ implies
\begin{eqnarray}\label{eq: K2 diff 22}
|t(K_2,W') - t(K_2,W) | \leq 2\epsilon \mu(U) + 2r^2 \eta^2 \leq
 3\epsilon \eta.
\end{eqnarray}

We now consider the increment in $K_r$-density.
Again, $U^{r,0}$ does not contribute to $t(K_r,W')-t(K_r,W)$. Hence
\begin{eqnarray}\label{eq: tKr tKrW}
& & \hspace{-1.7cm} t(K_r,W') - t(K_r,W) \nonumber \\
&=&  \int_{\mathbf{x} \in U^{r,1} } \prod_{ij\in \binom{[r]}{2}}  W(x_i,x_j) \Big(\prod_{i\in [r]} \dd \mu'(x_{i}) -  \prod_{i\in [r]}\dd \mu(x_{i})\Big)  \pm \mu^r ( U^{r,2}_+ )\nonumber
\\
&\stackrel{\eqref{eq: Omega 2 measure} }{=}&  r  \int_{x_1\in U} \int_{ U^{r-1,0} }  \prod_{ij \in \binom{[r]}{2}}  W(x_i,x_j) \Big(\prod_{i\in [r]} \dd \mu'(x_{i}) -\prod_{i\in [r]} \dd \mu(x_{i}) \Big)
\pm r^2  \eta^2 \nonumber \\
&=&   r \epsilon \left(\int_{x_1\in U_1 } - \int_{x_1\in U_2}\right)  \int_{ ([0,1]\setminus U)^{r-1} }  \prod_{ij\in \binom{[r]}{2}}   W(x_i,x_j) \prod_{i\in [r]} \dd \mu(x_{i})  
\pm  r^2\eta^2 \nonumber \\
&=&
r\epsilon \left(\int_{x_1\in U_1 } - \int_{x_1\in U_2}\right)   (t_{x_1}(K_r,W) \pm \mu^{r-1}(U^{r-1,1}_+) ) \dd \mu(x_1) \pm  r^2 \eta^2,
\end{eqnarray}
where the second equality holds by symmetry between the variables~$x_i$. By the definitions of $U_1$ and $U_2$, and by~\eqref{eq: def q f}--\eqref{eq: Omega 2 measure},
we further have
\begin{eqnarray*}
\eqref{eq: tKr tKrW} &\leq& r\epsilon\left(\int_{U_1}(q_r(x)-\gamma) - \int_{U_2} (q_r(x)+\gamma) \right)\dd \mu(x)  \pm  2r^2 \eta^2 \nonumber \\
 &\stackrel{\text{def of }q_r}{=}& r\epsilon   (r-1) (k-1)^{(r-2)} c^{r-2} \left(\int_{U_1} -\int_{U_2}\right)  d_W(x) \dd \mu(x) - r\gamma \epsilon \mu(U)  \pm 2r^2 \eta^2\nonumber \\
 &\leq & \epsilon r(r-1) (k-1)^{(r-2)} c^{r-2} \left(\int_{U_1} -\int_{U_2}\right)  d_W(x) \dd \mu(x) - r\gamma \epsilon \eta/2.
\end{eqnarray*}
Let $\alpha':= t(K_2,W')$.
Then, the above inequality, together with \eqref{eq: K2 diff}, Lemma~\ref{eq: h'} and~ \ref{W3}, implies that
\begin{eqnarray*}
t(K_r,W') &\leq& t(K_r,W) + \binom{r}{2} (k-1)^{(r-2)}c^{r-2} (\alpha'-\alpha \pm 2r^2 \eta^2) -r\gamma\epsilon\eta/2  \\
&<&  h_r(\alpha) +h'_r(\alpha) (\alpha' - \alpha) - r\gamma \epsilon \eta/3.
\end{eqnarray*}
On the other hand, by Lemma~\ref{lem: Taylor} and \eqref{eq: K2 diff 22}, we see that
$$h_r(\alpha') \geq h_r(\alpha) + h'_r(\alpha) (\alpha'-\alpha) - |\alpha'-\alpha|^{3/2}
> h_r(\alpha) +h'_r(\alpha) (\alpha' - \alpha) - \eta^{3/2}.$$
Hence, as $\eta \ll \epsilon$, we have 
$$t(K_r,W') <  h_r(\alpha) +h'_r(\alpha) (\alpha' - \alpha) - r\gamma \epsilon \eta/3 
< h_r(\alpha) +h'_r(\alpha) (\alpha' - \alpha) - \eta^{3/2} < h_r(\alpha').$$
This contradicts Theorem~\ref{thm: clique density}, proving the claim.
\end{proof}
We next show that the set $N_0$, which consists of ``$K_3$-heavy'' points, has positive measure.
\begin{claim}\label{cl: N0}
We have $\mu(N_{0} ) > 0$.
\end{claim}
\begin{proof}
For each $\gamma>0$, let $N_{\gamma}:=\{ x\in [0,1]: f_3(x) < -\gamma\}$. 
Let
$$\beta:= \frac{1}{2}(t(K_3,W)-h_3(\alpha))\stackrel{\ref{W2}}{>}0.$$
By Lemma~\ref{lem: integral f}, we have 
$\int_{[0,1]} f_3(x) \dd\mu(x) = h_3(\alpha) - t(K_3,W)=- 2\beta.$
On the other hand, we have  that, rather roughly, $f_3(x)\geq q_3(x)-1\ge -k^2$ for all $x\in [0,1]$. Thus 
we have
$$- 2\beta =\int_{[0,1]} f_3(x) \dd\mu(x) \geq - k^2\mu(N_{\beta}) - \beta (1-\mu(N_{\beta})),$$
implying that $\mu(N_{\beta}) \geq \beta/(k^2-\beta) >0$. Consequently, as $N_{\beta} \subseteq N_0$, we see that $\mu(N_0)\geq \mu(N_{\beta})>0$ as claimed.
\end{proof}

\subsection{Maximum degree condition}
We shall show in this subsection that almost every $x\in [0,1]$ satisfies $d_W(x)\leq kc$. For this, let
$$D:=\{ x\in [0,1] : d_W(x) > kc\}$$
be the set of points with ``too large degree''. We shall see that $D$ has measure zero. To show this we need one more statement about pairs that are ``$K_r$-heavy''.
Let $K_r^-$ denote the graph obtained from the complete graph on $[r]$ rooted at $1$ and $2$ by removing the edge $\{1,2\}$. Define
\begin{align}\label{eq: heavy B}
B_*:=\big\{ (x,y) \in [0,1]^2 \mid W(x,y)>0 \enspace \text{and} \enspace t_{x,y}(K_r^-,W)> (k-1)^{(r-2)} c^{r-2} \big\},
\end{align}
which one can think of as the set of edges that are ``$K_r$-heavy''.
As $W$ is Borel, the set $B_*$ is Borel. The following claim states that most of the pairs of ``adjacent'' points are not contained in too many copies of~$K_r$.
\begin{claim}\label{cl:2} 
We have $\mu^2(B_*) = 0$.
\end{claim}
\begin{proof} This claim also follows from Razborov's differential calculus~\cite[Corollary~4.6]{Razborov07}.
In terms of  graphs, the argument roughly says that if, on the contrary, $\Omega(n^2)$ edges of an almost extremal $(n,m)$-graph $G$ are each in too many copies of $K_r$ (namely, in at least $H_r(n,m)-H_r(n,m-1)+\Omega(n^{r-2})$ copies), then by removing a carefully selected subset of such edges we can destroy so many $r$-cliques so that the asymptotic result (Theorem~\ref{thm: clique density}) is violated.

Again, let us give a direct proof of the claim. Suppose $\mu^2(B_*)>0$.
For each $\epsilon>0$, let
$$B_{\epsilon}:= \{ (x,y) \in [0,1]^2 \mid W(x,y)> \epsilon \enspace \text{and} \enspace t_{x,y}(K_r^-,W)\geq  (k-1)^{(r-2)} c^{r-2} +\epsilon\}.$$ 
Note that $B_{\epsilon}$ is a Borel set since the function $(x,y)\mapsto t_{x,y}(K_r^-,W)$ is Borel by Tonelli's theorem.
As $\{B_{\epsilon}\}_{\epsilon>0}$ forms a collection of nested sets and $\bigcup_{\epsilon>0} B_{\epsilon} = B_*$,  there is $\epsilon>0$ such that the $\mu^2(B_\epsilon) \geq \epsilon$. 
We fix such $\epsilon>0$. 
By 
lowering the value of $\epsilon$ if necessary, and choosing a constant $\eta$, we assume that 
$$0<\eta \ll \epsilon \ll \mu^2(B), \alpha, 1/r, 1/k, c.$$
By Sierpinski's Theorem (Theorem~\ref{thm: nonatomic}), 
take a subset $B\subseteq B_{\epsilon}$ with $\mu^2(B)=\eta$. By Lemma~\ref{lem: shrinking}, there exists a symmetric Borel set $C\subseteq B$ satisfying \ref{B'1}--\ref{B'3}.

Define 
$$W'(x,y) := \left\{ \begin{array}{ll} (1-\epsilon)W(x,y), & \text{ if } (x,y) \in C, \\
W(x,y), & \text{ if } (x,y)\in [0,1]^2\setminus C.
\end{array}\right.
$$
As $C$ and $W$ are Borel, the function $W'$ is also Borel. 
Let $\alpha':= t(K_2,W')$.
As $\eta \ll \alpha$, we have
\begin{align}\label{eq: edge change}
\alpha'-\alpha= t(K_2,W')- t(K_2,W)= -\epsilon \int_{C} W(x_1,x_2)\ \prod_{i\in [2]} \dd\mu(x_{i}).
\end{align}
Since $\epsilon< W(x,y)\leq 1$ for all $(x,y)\in C \subseteq B_{\epsilon}$, we also have that 
\begin{align}\label{eq: edge change 2}
-\epsilon\mu^2(C) \leq \alpha'-\alpha \leq -\epsilon^2 \mu^2(C).
\end{align}

For each $\mathbf{x}= (x_1,x_2)\in C$, \ref{B'2} implies that
\begin{eqnarray}\label{eq: T1}
& &\hspace{-1cm}  
\int_{ (x_3,\dots, x_r) \in C^{r,1}(\mathbf{x}) }\prod_{ij\in \binom{[r]}{2}\setminus \{\{1,2\}\} }  W(x_i,x_j) \prod_{i=3}^{r} \dd \mu(x_{i})  \nonumber \\
&=& \left(\int_{ [0,1]^{r-2}  } \pm \int_{C^{r,2}_+(\mathbf{x})}\right)
\prod_{ij\in \binom{[r]}{2} \setminus \{\{1,2\}\} }  W(x_i,x_j) \prod_{i=3}^{r} \dd \mu(x_{i})  \nonumber \\
 &\stackrel{\ref{B'2}}{=}& t_{x_1,x_2}(K_r^{-},W)\pm 2r\eta.
\end{eqnarray}

Thus, by the symmetry between the variables $x_i$ and Tonelli's theorem,
we have
\begin{eqnarray}\label{eq: eq11}
& &\hspace{-1.5cm}  \int_{\mathbf{x}\in C^{r,1} } \prod_{ij\in \binom{[r]}{2} }  W(x_i,x_j)\ \dd\mu^r(\mathbf{x}) \nonumber \\
&=&
\binom{r}{2} \int_{(x_1,x_2) \in C } W(x_1,x_2) \int_{ (x_3,\dots,x_r) \in C^{r,1}(x_1,x_2)  }\prod_{ij\in \binom{[r]}{2}\setminus \{\{1,2\}\} }  W(x_i,x_j)  \prod_{i\in [r]} \dd \mu(x_i)  \nonumber\\
&\stackrel{\eqref{eq: T1}}{=}& \binom{r}{2} \int_{ (x_1,x_2)  \in C }W(x_1,x_2) \left( t_{x_1,x_2}(K_r^{-},W)\pm 2r\eta
\right) \prod_{i\in [2]}\dd \mu(x_i)
  \nonumber \\
&=& \hspace{-0.3cm}  \binom{r}{2} \left(\int_{  (x_1,x_2) \in C }W(x_1,x_2)  t_{x_1,x_2}(K_r^{-},W) \prod_{i\in [2]} \dd \mu(x_i)
\pm 2r\eta \mu^2(C)\right).
\end{eqnarray} 

We shall bound $K_r$-density in $W'$ in two ways to derive a contradiction. First, note that for all $2\le \ell\le {r\choose 2}$ and $(x_1,\dots, x_r)\in C^{r,\ell}$, we have that
$$\prod_{ij \in \binom{[r]}{2}}  W'(x_i,x_j) = (1-\epsilon)^\ell \prod_{ij \in \binom{[r]}{2}}  W(x_i,x_j).$$
As $C^{r,0}$ does not contribute to the change in $K_r$-density, we have
 \begin{eqnarray}\label{eq: large compute}
 & & \hspace{-1cm} t(K_r,W')-t(K_r,W)  \nonumber \\
 &=& -\epsilon \int_{\mathbf{x}\in C^{r,1} } \prod_{ ij\in \binom{[r]}{2}}  W(x_i,x_j)\ \dd\mu^r(\mathbf{x}) \pm \sum_{2\le \ell\le{r\choose 2}} (1-(1-\epsilon)^\ell) \int_{\mathbf{x} \in C^{r,\ell}} \prod_{ij\in \binom{[r]}{2}}  W(x_i,x_j)\ \dd\mu^r(\mathbf{x})  \nonumber\\
 &\stackrel{\eqref{eq: eq11}}{=}& 
- \epsilon \binom{r}{2} \left( \int_{  (x_1,x_2)  \in C }W(x_1,x_2)  t_{x_1,x_2}(K_r^{-},W) \prod_{i\in [2]} \dd \mu(x_i) \pm 2 r \eta \mu^2(C)
\right) \pm r^4 \epsilon \mu^r( C^{r,2}_+)
   \nonumber\\
   &\stackrel{\ref{B'3}}{\leq}& 
   - \epsilon \binom{r}{2} \left( \int_{ (x_1,x_2)\in C }W(x_1,x_2)  ((k-1)^{(r-2)} c^{r-2}+\epsilon)  \prod_{i\in [2]} \dd \mu(x_i)
\right) \pm \eta  \mu^2(C) \nonumber\\
 & \stackrel{\eqref{eq: edge change}}{=}&  \binom{r}{2} ((k-1)^{(r-2)} c^{r-2}+\epsilon) (\alpha'-\alpha) \pm   \eta  \mu^2(C) \nonumber\\
 &\stackrel{\text{Lem}~\ref{eq: h'}}{=}& h'_r(\alpha) (\alpha'-\alpha) + \binom{r}{2}\epsilon (\alpha'-\alpha) \pm \eta \mu^2(C) \nonumber\\
&\stackrel{\eqref{eq: edge change 2}}{\leq}& h'_r(\alpha)(\alpha'-\alpha)
- \epsilon^3 \mu^2(C)/2.
 \end{eqnarray}
 On the other hand, as $\eta \ll \epsilon \ll 1/k, 1/c, 1/r, \alpha$, 
 by Theorem~\ref{thm: clique density}, Lemma~\ref{lem: Taylor} and \ref{W3}, we have
 \begin{eqnarray*}
 t(K_r,W') &\geq& h_r(\alpha') \geq h_r(\alpha) + h'(\alpha)(\alpha'-\alpha) -  |\alpha'-\alpha|^{3/2}\\
 &\stackrel{\eqref{eq: edge change 2}}{\geq}& t(K_r,W) +  h'_r(\alpha)(\alpha'-\alpha) - \epsilon^{3/2} (\mu^2(C))^{3/2}  \\
&\stackrel{\ref{B'1}}{>}&  t(K_r,W) +  h'_r(\alpha)(\alpha'-\alpha)
- \epsilon^3 \mu^2(C)/2,
  \end{eqnarray*}
a contradiction to \eqref{eq: large compute}. This proves the claim.\end{proof}

We can now show that $D$, the set of ``large degree'' points, is negligible, thus imposing an additional ``maximum degree'' condition on our graphon.

\begin{claim}\label{cl: degree}
We have $\mu(D) =0$.
\end{claim}
\begin{proof} In graph theory language, the argument is informally as follows. Claim~\ref{cl:2} bounds the number of $r$-cliques per typical edge of an almost extremal graph $G$. This, by double counting, bounds the number of $r$-cliques per typical vertex $x$ in terms of its degree. On the other hand, the last two parameters are linearly related by Claim~\ref{cl:1}. Putting all together, we derive the claim.

Let us provide details. Recall the definition of $B_{*}$ in \eqref{eq: heavy B}.
For each $\gamma \geq 0$, let
$$B(\gamma):= \{ x\in [0,1] : 
\mu\left( \big\{ y\in [0,1]: (x,y) \in B_* \big\}\right) > \gamma\}.$$
By Tonelli's Theorem
and Claim~\ref{cl:2}, we have
$$0= \mu^2(B_*) = \int_{x\in [0,1]} \int_{y: (x,y)\in B_*} \dd\mu(x)\dd\mu(y) \geq \int_{x\in B(\gamma)} \gamma ~ \dd \mu(x) \geq \gamma \mu(B(\gamma)).$$
Thus we have $\mu(B(\gamma))=0$ for all $\gamma>0$ and, by Lemma~\ref{lm:CM},
$\mu(B(0))= \mu(\cup_{\gamma>0} B(\gamma))=0$.
Hence, by Claim~\ref{cl:1} it suffices to prove $D\subseteq M_0\cup B(0)$.
By the definition of $M_0$, for each $x\notin M_0\cup B(0)$, we have $f_r(x)= q_r(x)-t_x(K_r,W)=0$. Tonelli's theorem
then implies that
\begin{eqnarray*}
q_r(x)&=& (r-1) (d_W(x)-(k-1)c) (k-1)^{(r-2)}c^{r-2} + (k-1)^{(r-1)}c^{r-1} \\
&= & t_x(K_r,W) = \int_{y\in \Omega} t_{x,y} (K^-_r,W)W(x,y) \dd\mu(y) \\
&=& \left(\int_{y : (x,y)\in B_*} + \int_{y: (x,y)\notin B_* } \right)t_{x,y} (K^-_r,W)W(x,y) \dd\mu(y) 
 \\
 &\stackrel{\eqref{eq: heavy B}}{ \leq} & \mu(\{y\in \Omega: (x,y)\in B_*\} ) + (k-1)^{(r-2)}c^{r-2} d_W(x)\\
 &=& (k-1)^{(r-2)}c^{r-2} d_W(x),
 \end{eqnarray*}
 where the final inequality follows from the assumption $x\notin B(0)$.
 Rearranging this, we obtain 
 $$d_W(x) \leq kc,$$
 showing that $x\notin D$.
 Hence, $D\subseteq M_0\cup B(0)$ as claimed.
\end{proof}

\subsection{Putting everything together}
We are now ready to derive the final contradiction.
By Claims~\ref{cl:1},~\ref{cl: N0}~and~\ref{cl: degree}, we have $\mu(N_{0} \setminus (D\cup M_0) ) >0$. 
Fix a point 
$$x\in N_{0} \setminus (D\cup M_0).$$

By Tonelli's Theorem, the function $W(x,\cdot)$ is Borel.
For brevity, set $d:= d_W(x)$.
Note that $f_r(x)= q _r(x)-t_x(K_r,W) =0$ as $x\notin M_0$.

Suppose first that $d=0$. Then we have 
\begin{align*}
t_x(K_r,W) &= q_r(x) =- (r-1)(k-1)\cdot (k-1)^{(r-2)}c^{r-1} + (k-1)^{(r-1)}c^{r-1}  \\
&= - (r-2)k (k-1)^{(r-2)}c^{r-1}<0,
\end{align*}
a contradiction. Thus, we may assume that $d>0$.

Let $\tau:= c/d$. 
As $x\notin D$, we have
\begin{align}\label{eq: tau}
\tau \geq\frac{1}{k}.
\end{align}

Consider $W':= N_W(x)$, the neighbourhood of $x$ in $W$ as in Definition~\ref{def: nbrhd}. As $x\in N_0$, we have $f_3(x)=q_3(x)-t_x(K_3,W)<0$. 
We then derive from~\eqref{eq: nhd relation} and~\ref{W2} that
\begin{align}\label{eq: nhd alpha'}
\alpha' &:= t(K_2,W') =\frac{ t_x(K_3,W)}{d^2}
> \frac{q_3(x)}{d^2}   
= \frac{ 2(d-(k-1)c)(k-1)c + (k-1)(k-2)c^2 }{d^2}  \nonumber \\
&= 2(k-1)\tau - k(k-1) \tau^2.
\end{align}
Further define
\begin{align}\label{eq:  rho upper}
\rho:= 2(k-1)\tau- k(k-1)\tau^2 = 1-\frac{1}{k} -k(k-1)\left(\tau- \frac{1}{k}\right)^2  \leq 1- \frac{1}{k}.
\end{align}

Let us briefly overview where we stand in the proof now. Here, $\alpha'$ is the edge density of $N_W(x)$. Also, $\rho$ is the edge density of the neighbourhood of a vertex of degree $d=d_W(x)$ in an $r$-extremal graphon of the overall edge density $\alpha$, provided that this degree $d$ is realisable. The relation $f_r(x)=0$ implies that the density of $K_{r-1}$ in $N_W(x)$ is as "expected", see~\eqref{eq: Kr-1W'} below. In order to derive a contradiction to $\rho<\alpha$
we also need to exclude the case that $\rho$ is in the interior of the region when $h_{r-1}$ is zero.

\begin{claim} $\rho\ge 1-\frac1{r-2}$.
\end{claim}

\begin{proof} Since $f_r(x)=0$, we have that $q_r(x)=t_x(K_r,W)$ is non-negative. 
This implies via an easy calculation that $d=d(x)$ is at least $\frac{r-2}{r-1}\, ck$. In turn, this and~\eqref{eq: tau} give that $\tau=c/d$ lies between $\frac1k$ and $\frac{r-1}{(r-2)k}$. Since $\rho$ is a concave quadratic function of $\tau$, it is 
enough to verify that $\rho-(1-\frac1{r-2})$ is non-negative for these end-points. Routine calculations give respectively $\frac1{r-2}-\frac1k$ and $\frac{(k+1-r)(r-3)}{k(r-2)^2}$, both of which are non-negative as $k\ge r-2$ (by $h_r(\alpha)>0$).\end{proof}

Recall that $t_x(K_r,W)=q_r(x)$ and so
\eqref{eq: nhd relation} implies that
\begin{eqnarray}\label{eq: Kr-1W'}
t(K_{r-1},W') &=&\frac{ t_x(K_r,W)}{d^{r-1}} = \frac{ (r-1)(d-(k-1)c)(k-1)^{(r-2)}c^{r-2} + (k-1)^{(r-1)}c^{r-1}}{ d^{r-1}}\nonumber \\
&=& (r-1) (k-1)^{(r-2)} \tau^{r-2} - (r-2) k^{(r-1)} \tau^{r-1}. 
\end{eqnarray}
We can also deduce from the definition of $\rho$ that
$$
 (k-1)(k -1 - k \rho)   \stackrel{\eqref{eq:  rho upper}}{=}  \big((k-1) ( k\tau - 1)\big)^2.
$$
Here, the left-hand side is exactly the expression that appears under the square root when we define $p_{r-1,k-1}(\rho)$ in~\eqref{eq:prt}.  Thus we have by~\eqref{eq: tau} that
\begin{eqnarray*}
p_{r-1,k-1}(\rho)&= &\frac{(k-2)^{(r-3)}}{(k-1)^{r-2} k^{r-2} }  \left( k-1 + (k-1)(k\tau-1) \right)^{r-2} \left( k-1 - (r-2) (k-1) (k\tau -1 )\right)\\
&=&  (r-1)  (k-1)^{(r-2)} \tau^{r-2} - (r-2)  k^{(r-1)} \tau^{r-1} \\
&\stackrel{\eqref{eq: Kr-1W'}}{ =}& t(K_{r-1},W').
 \end{eqnarray*}

On the other hand, by \eqref{def: k} and \eqref{eq:  rho upper}, we have $k(\rho)\leq k-1$.
Recall that $\alpha'>\rho\ge 1-\frac1{r-2}$. Thus, by Lemma~\ref{lem: linear extension} and the fact that $h_{r-1}$ is a strictly increasing function
on $[1-\frac1{r-2},1]$, we have
$$h_{r-1}(\alpha') > h_{r-1}(\rho) = p_{r-1,k(\rho)}(\rho) \geq p_{r-1,k-1}(\rho) = t(K_{r-1},W').$$
Hence, we have $t(K_{r-1},W')< h_{r-1}(\alpha')$ while $t(K_2,W)=\alpha'$, a contradiction to Theorem~\ref{thm: clique density}. 

This completes the proof of Theorem~\ref{thm: main}.

\section*{Acknowledgements} 
The authors are grateful to the anonymous reviewer for carefully reading this paper.

\bibliographystyle{plain}



\begin{dajauthors}
\begin{authorinfo}[jaehoon]
  Jaehoon Kim\\
  Korea Advanced Institute of Sciences and Technology\\
  Daejeon, Republic of Korea\\
  jaehoon\imagedot{}kim\imageat{}kaist\imagedot{}ac\imagedot{}kr \\
  \url{https://sites.google.com/view/jaehoon-kim/}
\end{authorinfo}
\begin{authorinfo}[hong]
  Hong Liu\\
  Mathematics Institute and DIMAP\\
  University of Warwick\\
  Coventry, UK\\
  h\imagedot{}liu\imagedot{}9\imageat{}warwick\imagedot{}ac\imagedot{}uk \\
  \url{http://homepages.warwick.ac.uk/staff/H.Liu.9/}
\end{authorinfo}
\begin{authorinfo}[oleg]
  Oleg Pikhurko\\
  Mathematics Institute and DIMAP\\
  University of Warwick\\
  Coventry, UK\\
  o\imagedot{}pikhurko\imageat{}warwick\imagedot{}ac\imagedot{}uk\\
  \url{https://homepages.warwick.ac.uk/~maskat/}
\end{authorinfo}
\begin{authorinfo}[maryam]
  Maryam Sharifzadeh\\
  Department of Mathematics and Mathematical Statistics \\ Umeå University \\
Umeå, Sweden \\
maryam\imagedot{}sharifzadeh\imageat{}umu\imagedot{}se
\end{authorinfo}
\end{dajauthors}

\end{document}